\documentclass[smallextended,referee,envcountsect]{svjour3}
\usepackage [latin1]{inputenc}
\usepackage{amsmath,amssymb}
\usepackage{marvosym,mathtools}
\usepackage[numbers,sort&compress]{natbib}
\usepackage[colorlinks,linkcolor=blue,urlcolor=blue,citecolor=blue]{hyperref}
\usepackage{graphicx,subfig}
\usepackage{float}
\usepackage{epstopdf}
\usepackage{algorithm}
\usepackage{algorithmicx}    
\usepackage{algpseudocode}   
\usepackage{tcolorbox}

\smartqed
\usepackage{graphicx}
\journalname{}

\usepackage{fancyhdr}
\usepackage{ntheorem}
\theoremheaderfont{\bfseries\upshape}
\theorembodyfont{\upshape}
\renewtheorem{remark}{\it Remark}[section]
\renewtheorem{example}{Example}[section]

\begin{document}

\title{Inertial accelerated primal-dual algorithms for non-smooth convex optimization problems with linear  equality constraints}

\author{Huan Zhang$^ {1}$ \and Xiangkai Sun$^{2}$ \and Shengjie Li$^{1}$ \and Kok Lay Teo$^{3}$}

\institute{\\Huan Zhang \at{\small zhanghwxy@163.com} \\
           \\ Xiangkai Sun  \at {\small sunxk@ctbu.edu.cn}\\
           \\Shengjie Li (\Letter) \at{\small lisj@cqu.edu.cn}\\
           \\ Kok Lay Teo \at{\small K.L.Teo@curtin.edu.au}\\\\
$^{1}$College of Mathematics and Statistics, Chongqing University, Chongqing 401331, China\\\\
$^{2}$Chongqing Key Laboratory of Social Economy and Applied Statistics, College of Mathematics and Statistics, Chongqing Technology and Business University,
Chongqing 400067, China\\\\
$^{3}$School of Mathematical Sciences, Sunway University, Bandar Sunway, 47500 Selangor Darul Ehsan, Malaysia.}

\date{Received: date / Accepted: date}

\maketitle

\begin{abstract}
This paper is devoted to the study of an inertial accelerated primal-dual algorithm, which is based on a second-order differential system with time scaling, for solving a non-smooth convex optimization problem with linear equality constraints. We first introduce a second-order differential system with time scaling associated with the non-smooth convex optimization problem, and then obtain fast convergence rates for the primal-dual gap, the feasibility violation, and the objective residual along the trajectory generated by this system. Subsequently, based on the setting of the parameters involved, we propose an inertial accelerated primal-dual algorithm from the time discretization of this system. We also establish fast convergence rates for the primal-dual gap, the feasibility violation, and the objective residual. Furthermore, we demonstrate the efficacy of the proposed algorithm through numerical experiments.
\end{abstract}

\section{Introduction}

Let $\mathcal{H}$ and $\mathcal{G}$ be real Hilbert spaces. Let $f:\mathcal{H}\rightarrow \mathbb{R}$ be a differentiable convex function. The convex optimization problem with linear equality constraints is defined as
\begin{eqnarray}\label{non}
\left\{ \begin{split}
&\mathop{\mbox{min}}\limits_{x\in\mathcal{H}}~~{f(x)}\\
&\mbox{s.t.}~~Ax=b,
\end{split}
\right.
\end{eqnarray}
where $A:\mathcal{H}\rightarrow\mathcal{G}$ is a continuous linear operator and $b\in\mathcal{G}$. The optimization problem of  type (\ref{non}) is an important model and has been  used in a wide range of fields, such as image recovery, machine learning, and network optimization, see \cite{boyd2010,gold2014,ouyang2015,li2019} and the references therein.

Nowadays, primal-dual dynamical systems offer a powerful framework for analyzing optimization algorithms designed to solve Problem (\ref{non}). It provides both theoretical insights into existing optimization algorithms and practical tools, such as Lyapunov theory, to simplify convergence analysis. In recent years, the use of primal-dual dynamical systems featuring viscous damping, time scaling, and extrapolation coefficients has led to many successful results for solving Problem (\ref{non}) from various perspectives. Bo\c{t} and Nguyen \cite{bot2021n} first propose the following second-order primal-dual  dynamical system
\begin{equation}\label{bot}
\left\{ \begin{split}
&\ddot{x}(t)+\frac{\alpha}{t}\dot{x}(t)+\nabla_x \mathcal{L}_{\sigma} \left(x(t),\lambda(t) + \theta t\dot{\lambda}(t)\right) =0,\\
&\ddot{\lambda}(t)+\frac{\alpha}{t} \dot{\lambda}(t)-\nabla_{\lambda} \mathcal{L}_{\sigma} \left(x(t)+\theta t \dot{x}(t),\lambda(t)\right)=0,
\end{split}
\right.
\end{equation}
where $t \geq t_0 >0$, $\alpha>3$, $ \sigma \geq 0$, $\frac{1}{2} \geq \theta \geq \frac{1}{\alpha-1}$, and $\mathcal{L}_{\sigma}(x,\lambda) \coloneqq f(x)+\left \langle \lambda,Ax-b \right \rangle+\frac{\sigma}{2}\|Ax-b\|^2$. They obtain the $\mathcal{O} \left(\frac{1}{t^2} \right)$ convergence rate for the primal-dual gap and feasibility violation, and show that the trajectory generated by System (\ref{bot}) converges weakly to an optimal solution of Problem (\ref{non}).
Zeng et al. \cite{zeng2023} establish the $\mathcal{O} \left(t^{-\frac{2\alpha}{3}} \right)$ convergence rate for the primal-dual gap and feasibility violation along the trajectory generated by System (\ref{bot}) with $ 0 <\alpha \leq 3$, $\sigma=0$, and $\theta=\frac{3}{2\alpha}$.
Subsequently, Hulett and Nguyen \cite{hn2023} propose the following second-order primal-dual dynamical system with time scaling
\begin{equation}\label{hn}
\left\{ \begin{split}
&\ddot{x}(t)+\frac{\alpha}{t}\dot{x}(t)+\beta(t)\nabla_x \mathcal{L}_{\sigma} \left(x(t),\lambda(t) + \theta t\dot{\lambda}(t)\right) =0,\\
&\ddot{\lambda}(t)+\frac{\alpha}{t} \dot{\lambda}(t)-\beta(t)\nabla_{\lambda} \mathcal{L}_{\sigma} \left(x(t)+\theta t \dot{x}(t),\lambda(t)\right)=0,
\end{split}
\right.
\end{equation}
where $\beta : [t_0, +\infty) \rightarrow \mathbb{R}$ is a time scaling function.  They establish the $\mathcal{O} \left(\frac{1}{t^2 \beta(t)} \right)$ convergence rate for the primal-dual gap, the feasibility violation, and the objective residual along the trajectory generated by System (\ref{hn}).
In order to improve convergence rates, He et al. \cite{hehufang2022} propose the following ``second-order'' primal + ``first-order'' dual dynamical system
\begin{equation}\label{he}
\left\{ \begin{split}
&\ddot{x}(t)+r \dot{x}(t)+\beta(t)\nabla_x \mathcal{L}_{\sigma} \left(x(t),\lambda(t)\right) =0,\\
&\dot{\lambda}(t)-\beta(t) \nabla_{\lambda} \mathcal{L}_{\sigma} \left(x(t)+\theta \dot{x}(t),\lambda(t)\right)=0,
\end{split}
\right.
\end{equation}
where $r\geq 0$ and $\theta \geq 0$. They obtain the $\mathcal{O} \left(\frac{1}{e^{r t}} \right)$ exponential convergence rate for the primal-dual gap, the feasibility violation, and the objective residual.
For more details on continuous-time primal-dual dynamical systems for solving Problem (\ref{non}), we refer to \cite{h2021,he2023,2024zhu,hetian2025,li2025,s2025} and the references therein.

Recently, many researchers propose different kinds of  inertial accelerated algorithms based on the time discretization of primal-dual dynamical systems for solving Problem (\ref{non}).
Bo\c{t} et al. \cite{bot2023} propose a fast augmented Lagrangian algorithm from the time discretization of System (\ref{bot}) and obtain the $\mathcal{O} \left( \frac{1}{ k^2} \right)$ convergence rate for the primal-dual gap, the feasibility violation, and the objective residual.
By the time discretization of System (\ref{hn}) with $\frac{\alpha}{t}=r \geq 0$ and $\theta t=\delta >0$, Ding et al. \cite{dingliu2025} propose an inertial algorithm and obtain the $\mathcal{O} \left( \frac{1}{ \beta_k} \right)$ convergence rate for the primal-dual gap, the feasibility violation, and the objective residual.
He et al. \cite{hehu2022} design a non-ergodic primal-dual algorithm with
the $\mathcal{O} \left( \frac{1}{ k^{\alpha-1}} \right)$ convergence rate in both the feasibility violation and the objective residual, tailored to the time discretization of System (\ref{he}) with $r=\frac{\alpha}{t}$, $\sigma=0$, and $\theta=\frac{1}{\alpha-1}$, $\forall \alpha > 1$.
By the time discretization of a Tikhonov regularized primal-dual dynamical system introduced in \cite{2024zhu}, Zhu et al. \cite{zhufang2025} propose a fast primal-dual algorithm for solving Problem (\ref{non}), and also obtain the $\mathcal{O} \left( \frac{1}{ k^2} \right)$ convergence rate for the primal-dual gap, the feasibility violation, and the objective residual.

We observe that there is limited research on inertial accelerated algorithms from a dynamical system perspective for solving Problem (\ref{non}) where objective functions have a composite structure, although there are some preliminary results available in \cite{hehuang2025,l2025} on accelerated algorithms for the unconstrained optimization problem. Therefore, by the time discretization of second-order differential systems, we will introduce a new inertial accelerated primal-dual algorithm to solve the non-smooth optimization problem
\begin{eqnarray}\label{constrained}
\left\{ \begin{split}
&\mathop{\mbox{min}}\limits_{x\in\mathcal{H}}~~{f(x)+g(x)}\\
&\mbox{s.t.}~~Ax=b,
\end{split}
\right.
\end{eqnarray}
where $f: \mathcal{H}\rightarrow \mathbb{R}$ is a differentiable function and its gradient $\nabla f$ is Lipschitz continuous, and $g: \mathcal{H}\rightarrow \mathbb{R}$ is a proper convex and lower semi-continuous function.
It is worth noting that many practical problems in various fields can be modeled as Problem (\ref{constrained}), such as image restoration, support vector machine, and sparse portfolio optimization problems \cite{d2012,t2017,k2020,t2020}. Moreover, several algorithms have been developed for solving Problem (\ref{constrained}), including the random coordinate descent algorithm \cite{n2014}, augmented Lagrangian algorithms \cite{xu2017,d2023}, and primal-dual algorithms \cite{he2022,jiang2024}.

In this paper, we first introduce the following second-order differential system, which consists of viscous damping, extrapolation and time scaling,
\begin{equation}\label{dyn}
\left\{ \begin{split}
&\ddot{x}(t)+\frac{\alpha}{t}\dot{x}(t)+\beta(t)\partial_{x}\mathcal{L}_\rho\left(x(t),\lambda(t)
+\frac{t}{\alpha-1}\dot{\lambda}(t)\right)\ni0,\\
&\ddot{\lambda}(t)+\frac{\alpha}{t}\dot{\lambda}(t)-\beta(t)\nabla_\lambda \mathcal{L}_\rho\left(x(t)+\frac{t}{\alpha-1} \dot{x}(t),\lambda(t)\right)=0,
\end{split}
\right.
\end{equation}
where $t\geq t_0>0$, $ \mathcal{L}_\rho : \mathcal{H} \times \mathcal{G} \to \mathbb{R}$ is the augmented Lagrangian saddle function (see  (\ref{asd})  for details), $\frac{\alpha}{t}$ is the viscous damping parameter, $\beta:[t_0,+\infty)\rightarrow(0,+\infty)$ is the time scaling function which is  non-decreasing and continuously differentiable, and  $ \frac{t}{\alpha-1} $ is the extrapolation parameter. Then, we propose an inertial accelerated primal-dual algorithm by discretizing System (\ref{dyn}), for solving Problem (\ref{constrained}), and give some convergence analysis.
The contributions of this paper can be more specifically stated as follows:
\begin{itemize}
\item[{\rm (i)}] We propose a new second-order differential system (\ref{dyn}) with time scaling for solving the non-smooth optimization problem (\ref{constrained}). Compared with the systems in \cite{bot2023}, System (\ref{dyn}) incorporates viscous damping, extrapolation and time scaling.

\item[{\rm (ii)}]  Under mild assumptions on the parameters, we show that the primal-dual gap along the trajectories generated by System (\ref{dyn})  enjoys the $ \mathcal{O}\left(\frac{1}{t^{2} \beta(t) }\right) $ convergence rate. We also show that the feasibility violation and the objective residual enjoy the
    $ \mathcal{O}\left(\frac{1}{t \sqrt{\beta(t)} }\right) $ convergence rate.

\item[{\rm (iii)}] By appropriately adjusting these parameters, we show that the inertial accelerated primal-dual algorithm proposed in this paper exhibits the $ \mathcal{O}\left(\frac{1}{k^{2} \beta_{k} }\right) $ convergence rate for the primal-dual gap, the feasibility violation, and the objective residual.

\item[{\rm (iv)}] Through numerical experiments, we demonstrate that the inertial accelerated primal-dual algorithm generated by System (\ref{dyn}), controlled by time scaling during iterations, can help accelerate the convergence.
\end{itemize}

The rest of this paper is organized as follows. In Section 2,  we recall some basic notations and present some preliminary results. In Section 3, we  obtain the fast convergence rates of the primal-dual gap, the feasibility violation, and the objective residual along the trajectories generated by System (\ref{dyn}), and also give some integral estimate results.  In Section 4, we propose an inertial accelerated primal-dual algorithm for solving Problem (\ref{constrained}) and establish fast convergence rates for the primal-dual gap, the feasibility violation, and the objective residual. In Section 5, we present some numerical experiments to illustrate the obtained results.

\section{Preliminaries}
Unless otherwise specified,  let  $\mathcal{H}$ and $\mathcal{G}$ be real Hilbert spaces equipped  with inner product $ \langle \cdot, \cdot \rangle$ and norm $\| \cdot \|$. The norm of the Cartesian
 product $ \mathcal{H} \times \mathcal{G} $ is defined as
\begin{equation*}
\|(x,y)\|=\sqrt{\|x\|^2+\|y\|^2}, ~~~ \forall (x,y)\in\mathcal{H}\times\mathcal{G}.
\end{equation*}
For every $x_1,x_2\in \mathcal{H}$, the following equality holds:
\begin{equation}\label{o1}
\frac{1}{2}\|x_1\|^2-\frac{1}{2}\|x_2\|^2=\langle x_1,x_1-x_2 \rangle-\frac{1}{2} \|x_1-x_2\|^2.
\end{equation}

The Lagrangian function associated with Problem (\ref{constrained}) is defined as
\begin{equation*}\label{PD}
\mathcal{L}(x,\lambda):= f(x)+g(x)+\left \langle \lambda,Ax-b \right \rangle.
\end{equation*}
A pair $ (x^*,\lambda^*) \in \mathcal{H}\times \mathcal{G}$ is said to be a saddle point of the  Lagrangian function $ \mathcal{L} $ iff
\begin{equation*}
\mathcal{L}(x^*,\lambda)\leq \mathcal{L}(x^*,\lambda^*) \leq \mathcal{L}(x,\lambda^*),~~\forall (x,\lambda)\in\mathcal{H}\times\mathcal{G}.
\end{equation*}

For $\rho > 0 $, associated with the Lagrangian function $ \mathcal{L} $, we introduce the augmented Lagrangian function $ \mathcal{L}_\rho: \mathcal{H} \times \mathcal{G} \to \mathbb{R}$ defined by
\begin{equation}\label{asd}
\mathcal{L}_\rho(x,\lambda):= \mathcal{L}(x,\lambda)  + \frac{\rho}{2}\|Ax-b\|^2=f(x)+g(x)+\left \langle \lambda,Ax-b \right \rangle+\frac{\rho}{2}\|Ax-b\|^2.
\end{equation}
In the sequel, the set of saddle points of $\mathcal{L}_\rho$ is denoted by $\mathcal{S}$. The set of feasible points of Problem (\ref{constrained}) is denoted by $\mathcal{F}:=\{ x\in \mathcal{H} | A x=b \}$.  For any $(x,\lambda)\in \mathcal{F} \times \mathcal{G}$, it holds that $f(x)+g(x)=\mathcal{L}_{\rho}(x,\lambda)=\mathcal{L}(x,\lambda)$. We assume that $\mathcal{S}\neq\emptyset$. Let $(x^*,\lambda^*)\in\mathcal{S}$. Then,
\begin{equation*}
(x^*,\lambda^*)\in \mathcal{S} \Leftrightarrow
\left\{ \begin{split}
&0 \in \partial_x\mathcal{L}_{\rho}(x^*,\lambda^*)=\nabla f(x^*)+ \partial  g(x^*)+A^*\lambda^*,\\
&0=\nabla_\lambda\mathcal{L}_{\rho}(x^*,\lambda^*)=Ax^*-b,
\end{split}
\right.
\end{equation*}
where $A^*:\mathcal{G}\rightarrow\mathcal{H}$ denotes the adjoint operator of $A$.

\section{Fast Convergence Rates for Differential System \textup{(\ref{dyn})}}
In this section, by using the Lyapunov analysis, we establish  the fast convergence rates for the primal-dual gap, the feasibility violation, and the objective residual along the trajectory generated by System \textup{(\ref{dyn})} under mild assumptions on the parameters. Moreover, we also give some integral estimates.

\begin{theorem}\label{Th1}
Let $ (x ,\lambda ):\left[t_0, +\infty\right)\to \mathcal{H} \times \mathcal{G}$ be a solution of System $ (\ref{dyn})$.
Suppose that $\alpha \geq 3$, $\rho>0$ and $\sup\limits_{t\geq t_0}\frac{t\dot{\beta}(t)}{\beta(t)}\leq 3-\alpha$.
Then, for any $ (x^*,\lambda^*)\in \mathcal{S} $,
\begin{equation*}
\mathcal{L}_{\rho}(x(t),\lambda^*)-\mathcal{L}_{\rho}(x^*,\lambda^*)=\mathcal{O}\left(\frac{1}{t^{2}\beta(t)}\right),~\textup{as} ~  t\to +\infty ,
\end{equation*}
\begin{equation*}\resizebox{0.97\hsize}{!}{$
\|Ax(t)-b\|=\mathcal{O}\left( \frac{1}{ t \sqrt{\beta(t)}}\right)
\textup{ and }
|(f+g)(x(t))-(f+g)(x^*)|=\mathcal{O}\left( \frac{1}{ t \sqrt{\beta(t)}}\right),~\textup{as} ~ t\to +\infty,$}
\end{equation*}
\begin{equation*}
\int_{t_0}^{+\infty} t\beta(t)\left\| Ax(t)-b\right\|^2dt< +\infty,
\end{equation*}
and
\begin{equation*}
\int_{t_0}^{+\infty} \left((\alpha-3)t\beta(t)-t^2\dot{\beta}(t)\right)
\left( \mathcal{L}_\rho(x(t),\lambda^*)-\mathcal{L}_\rho(x^*,\lambda^*)\right) dt< +\infty.
\end{equation*}
\end{theorem}
\begin{proof}

From the Moreau-Yosida regularization \cite[7.c]{m1965}, the smoothing approximation $g_{\gamma} (x)$ of $g(x)$ is defined by
$$
g_{\gamma} (x):=\inf_{u\in \mathcal{H}} \left\{ g(u)+\frac{1}{2\gamma}\|x-u\|^2 \right\}.
$$
Obviously, $g_{\gamma} (x)$ is continuously differentiable and $\nabla g_{\gamma} (x)$ is Lipschitz continuous. Now, we consider the following system, which is consistent with System (\ref{dyn}),
\begin{equation}\label{w1}
\left\{\begin{split}
&\ddot{x}_{\gamma}(t)+\frac{\alpha}{t}\dot{x}_{\gamma}(t)+\beta(t)\Big(\nabla f\left(x_{\gamma}(t)\right)+\nabla g_{\gamma}\left(x_{\gamma}(t)\right)\\
&~~~~+A^*\left(\lambda_{\gamma}(t)+\frac{t}{\alpha-1} \dot{\lambda}_{\gamma}(t)\right)+\rho A^*\left(Ax_{\gamma}(t)-b\right)\Big)= 0,\\
&\ddot{\lambda}_{\gamma}(t)+\frac{\alpha}{t}\dot{\lambda}_{\gamma}(t)-\beta(t)\left(A \left(x_{\gamma}(t)+\frac{t}{\alpha-1} \dot{x}_{\gamma}(t)\right)-b\right)=0.
\end{split}
\right.
\end{equation}
For any fixed $ (x^*,\lambda^*)\in \mathcal{S} $, we introduce the energy function $ \mathcal{E}_{\gamma}:\left[t_0,+\infty \right) \to \mathbb{R} $ which defined as
\begin{equation}\label{def1}
\begin{aligned}
\mathcal{E}_{\gamma}(t)=&\frac{t^2 \beta(t)}{(\alpha-1)^2} \left(\mathcal{L}_{\rho}(x_{\gamma}(t),\lambda^*)-\mathcal{L}_{\rho}(x^*,\lambda^*)\right)\\
&+\frac{1}{2}\lVert x_{\gamma}(t)-x^*+\frac{t}{\alpha-1}\dot{x}_{\gamma}(t)\rVert^2
+\frac{1}{2}\lVert \lambda_{\gamma}(t)-\lambda^*+\frac{t}{\alpha-1}\dot{\lambda}_{\gamma}(t)\rVert^2.
\end{aligned}
\end{equation}
Clearly, $ \mathcal{E}_{\gamma}(t) \geq0 $  for all $t\geq t_0 $.
Then,
\begin{equation*}\label{a1}{
\begin{split}
\dot{\mathcal{E}}_{\gamma}(t)
=&\left(\frac{2t\beta(t)}{(\alpha-1)^2}+\frac{t^2\dot{\beta}(t)}{(\alpha-1)^2}\right)\left(\mathcal{L}_{\rho} \left(x_{\gamma}(t),\lambda^*\right)-\mathcal{L}_{\rho}(x^*,\lambda^*)\right)\\
&+\frac{t^2\beta(t)}{(\alpha-1)^2}\langle \nabla_x \mathcal{L}_{\rho}(x_{\gamma}(t),\lambda^*),\dot{x}_{\gamma}(t)  \rangle
\\
&+\left\langle x_{\gamma}(t)-x^*+\frac{t}{\alpha-1} \dot{x}_{\gamma}(t),\frac{\alpha}{\alpha-1}\dot{x}_{\gamma}(t)+\frac{t}{\alpha-1} \ddot{x}_{\gamma}(t) \right\rangle\\
&+\left\langle \lambda_{\gamma}(t)-\lambda^*+\frac{t}{\alpha-1} \dot{\lambda}_{\gamma}(t),\frac{\alpha}{\alpha-1}\dot{\lambda}_{\gamma}(t)+\frac{t}{\alpha-1} \ddot{\lambda}_{\gamma}(t) \right\rangle\\
=&\left(\frac{2t\beta(t)}{(\alpha-1)^2}+\frac{t^2\dot{\beta}(t)}{(\alpha-1)^2}\right)\left(\mathcal{L}_{\rho}\left(x_{\gamma}(t),\lambda^*\right)-\mathcal{L}_{\rho}(x^*,\lambda^*)\right)\\
&+\frac{t^2\beta(t)}{(\alpha-1)^2}\left\langle \nabla f\left(x_{\gamma}(t)\right)+\nabla g_{\gamma}\left(x_{\gamma}(t)\right),\dot{x}_{\gamma}(t)  \right\rangle
+\frac{t^2\beta(t)}{(\alpha-1)^2}\left\langle \lambda^* , A \dot{x}_{\gamma}(t) \right\rangle\\
&+ \frac{\rho t^2\beta(t)}{(\alpha-1)^2}\left\langle A x_{\gamma}(t)-b, A \dot{x}_{\gamma}(t)  \right\rangle\\
&+ \left\langle x_{\gamma}(t)-x^*+ \frac{t}{\alpha-1} \dot{x}_{\gamma}(t),\frac{\alpha}{\alpha-1}\dot{x}_{\gamma}(t)+\frac{t}{\alpha-1} \ddot{x}_{\gamma}(t)\right\rangle\\
&+\left\langle \lambda_{\gamma}(t)-\lambda^*+\frac{t}{\alpha-1} \dot{\lambda}_{\gamma}(t),\frac{\alpha}{\alpha-1}\dot{\lambda}_{\gamma}(t)+\frac{t}{\alpha-1} \ddot{\lambda}_{\gamma}(t) \right\rangle.
\end{split}}
\end{equation*}
Note that
\begin{equation*}
\begin{aligned}
&\left\langle x_{\gamma}(t)-x^*+\frac{t}{\alpha-1} \dot{x}_{\gamma}(t),\frac{\alpha}{\alpha-1}\dot{x}_{\gamma}(t)+\frac{t}{\alpha-1} \ddot{x}_{\gamma}(t) \right\rangle\\
=&-\frac{t\beta(t)}{\alpha-1}\left\langle x_{\gamma}(t)-x^*+\frac{t}{\alpha-1}\dot{x}_{\gamma}(t),
\nabla f\left(x_{\gamma}(t)\right)+\nabla g_{\gamma}\left(x_{\gamma}(t)\right) \right\rangle\\
&-\frac{t\beta(t)}{\alpha-1}\left\langle x_{\gamma}(t)-x^*+\frac{t}{\alpha-1}\dot{x}_{\gamma}(t),A^* \left( \lambda_{\gamma}(t) + \frac{t}{\alpha-1} \dot{\lambda}_{\gamma}(t)  \right)+\rho A^* \left( A x_{\gamma}(t)-b \right) \right\rangle\\
=&-\frac{t\beta(t)}{\alpha-1}\left\langle x_{\gamma}(t)-x^*+\frac{t}{\alpha-1} \dot{x}_{\gamma}(t), \nabla f\left(x_{\gamma}(t)\right)+\nabla g_{\gamma}\left(x_{\gamma}(t)\right) \right\rangle\\
&-\frac{t\beta(t)}{\alpha-1}\left\langle A \left(x_{\gamma}(t)+\frac{t}{\alpha-1} \dot{x}_{\gamma}(t)\right)-b,\lambda_{\gamma}(t)+\frac{t}{\alpha-1} \dot{\lambda}_{\gamma}(t) \right\rangle \\
&- \frac{\rho t\beta(t)}{\alpha-1}\left\langle  A \left( x_{\gamma}(t)+\frac{t}{\alpha-1} \dot{x}_{\gamma}(t) \right)-b , A x_{\gamma}(t)-b \right\rangle,
\end{aligned}
\end{equation*}
and
\begin{equation*}
\begin{aligned}
&\left\langle \lambda_{\gamma}(t)-\lambda^*+\frac{t}{\alpha-1} \dot{\lambda}_{\gamma}(t),\frac{\alpha}{\alpha-1}\dot{\lambda}_{\gamma}(t)+\frac{t}{\alpha-1} \ddot{\lambda}_{\gamma}(t) \right\rangle\\
=&\frac{t\beta(t)}{\alpha-1}\left\langle \lambda_{\gamma}(t)-\lambda^*+\frac{t}{\alpha-1} \dot{\lambda}_{\gamma}(t),A\left(x_{\gamma}(t)+\frac{t}{\alpha-1} \dot{x}_{\gamma}(t)\right)-b \right\rangle.
\end{aligned}
\end{equation*}
Then,
\begin{equation}\label{a2}
\begin{aligned}
\dot{\mathcal{E}}_{\gamma}(t)=&\left(\frac{2t\beta(t)}{(\alpha-1)^2}+\frac{t^2\dot{\beta}(t)}{(\alpha-1)^2}\right)
\left(\mathcal{L}_{\rho}(x_{\gamma}(t),\lambda^*)-\mathcal{L}_{\rho}(x^*,\lambda^*)\right)\\
&-\frac{t\beta(t)}{\alpha-1}\left\langle x_{\gamma}(t)-x^*,\nabla f\left(x_{\gamma}(t)\right)+\nabla g_{\gamma}\left(x_{\gamma}(t)\right) \right\rangle\\
&-\frac{t\beta(t)}{\alpha-1}\left\langle \lambda^*, A x_{\gamma}(t)-b \right\rangle- \frac{\rho t\beta(t)}{\alpha-1}\left\| Ax_{\gamma}(t)-b \right\|^2\\
\leq & \left(\frac{2t\beta(t)}{(\alpha-1)^2}+\frac{t^2\dot{\beta}(t)}{(\alpha-1)^2}\right)\left(\mathcal{L}_{\rho}(x_{\gamma}(t),\lambda^*)-\mathcal{L}_{\rho}(x^*,\lambda^*)\right)\\
&-\frac{t\beta(t)}{\alpha-1}\left( f(x_{\gamma}(t))-f(x^*) \right)-\frac{t\beta(t)}{\alpha-1}\left( g_{\gamma}(x_{\gamma}(t))-g_{\gamma}(x^*) \right)\\
&-\frac{t\beta(t)}{\alpha-1}\left\langle \lambda^*, A x_{\gamma}(t)-b \right\rangle- \frac{\rho t\beta(t)}{\alpha-1}\left\| Ax_{\gamma}(t)-b \right\|^2\\
=&\left(\frac{(3-\alpha)t\beta(t)}{(\alpha-1)^2}+\frac{t^2\dot{\beta}(t)}{(\alpha-1)^2}\right)
\left(\mathcal{L}_{\rho}(x_{\gamma}(t),\lambda^*)-\mathcal{L}_{\rho}(x^*,\lambda^*)\right)\\
&- \frac{\rho t\beta(t)}{2(\alpha-1)} \left\| Ax_{\gamma}(t)-b \right\|^2,
\end{aligned}
\end{equation}
where the inequality holds due to the convexity of $f$ and $g_{\gamma}$.
Since  $\sup\limits_{t\geq t_0}\frac{t\dot{\beta}(t)}{\beta(t)}\leq \alpha-3$, we have $\frac{(3-\alpha)t\beta(t)}{(\alpha-1)^2}+\frac{t^2\dot{\beta}(t)}{(\alpha-1)^2}\leq 0$   for all $t\geq t_0$. Then, $\dot{\mathcal{E}}_{\gamma}(t)\leq 0$   for all $t\geq t_0$. This means that $\mathcal{E}_{\gamma}(t)\leq\mathcal{E}_{\gamma}(t_0)$    for all $t\geq t_0$.
For any $t\geq t_0$, it follows from integrating (\ref{a2}) from $t_0$ to $t$ that
\begin{equation*}
\begin{aligned}
&\mathcal{E}_{\gamma}(t)+\int_{t_0}^t\left(\frac{(\alpha-3)s\beta(s)}{(\alpha-1)^2}-\frac{s^2\dot{\beta}(s)}{(\alpha-1)^2}\right)
\left( {\cal L}_\rho(x_{\gamma}(s),\lambda^*)-\mathcal{L}_\rho(x^*,\lambda^*)\right) ds\\
&+\int_{t_0}^t \frac{\rho s\beta(s)}{2(\alpha-1)}
\left\| Ax_{\gamma}(s)-b\right\|^2ds\leq\mathcal{E}_{\gamma}(t_0).
\end{aligned}
\end{equation*}
Thus,
\begin{eqnarray*}
\int_{t_0}^{+\infty} \left((\alpha-3)t\beta(t)-t^2\dot{\beta}(t)\right)
\left(\mathcal{L}_\rho(x_{\gamma}(t),\lambda^*)-\mathcal{L}_\rho(x^*,\lambda^*)\right) dt \leq \mathcal{E}_{\gamma}(t_0)\leq +\infty,\\
\int_{t_0}^{+\infty} t\beta(t)\left\| Ax_{\gamma}(t)-b\right\|^2dt \leq \mathcal{E}_{\gamma}(t_0)< +\infty.
\end{eqnarray*}
Moreover, by (\ref{def1}), we have
\begin{equation*}
t^2\beta(t)\left(\mathcal{L}_{\rho}(x_{\gamma}(t),\lambda^*)-\mathcal{L}_{\rho}(x^*,\lambda^*)\right)\leq \mathcal{E}_{\gamma}(t) \leq \mathcal{E}_{\gamma}(t_0), ~\forall ~ t\geq t_0.
\end{equation*}
This implies
\begin{equation*}
\mathcal{L}_{\rho}(x_{\gamma}(t),\lambda^*)-\mathcal{L}_{\rho}(x^*,\lambda^*)=\mathcal{O}\left( \frac{1}{ t^2 \beta(t)}\right),~\textup{as} ~ t\to +\infty.
\end{equation*}
Note that
$$
\mathcal{L}_{\rho}(x_{\gamma}(t),\lambda^*)-\mathcal{L}_{\rho}(x^*,\lambda^*)
=\mathcal{L}(x_{\gamma}(t),\lambda^*)-\mathcal{L}(x^*,\lambda^*)+\frac{\rho}{2}\|Ax_{\gamma}(t)-b\|^2.
$$
Then,
\begin{equation*}
\|Ax_{\gamma}(t)-b\|=\mathcal{O}\left( \frac{1}{ t \sqrt{\beta(t)}}\right), ~\textup{as} ~t\rightarrow +\infty.
\end{equation*}
Combining $\rho >0$ and the definition of $\mathcal{L}_{\rho}$, we obtain
\begin{equation*}
\begin{split}
&|(f+g)(x_{\gamma}(t))-(f+g)(x^*)|\\
\leq&
\mathcal{L}_{\rho}(x_{\gamma}(t),\lambda^*)-\mathcal{L}_{\rho}(x^*,\lambda^*)+|\langle \lambda^*, Ax_{\gamma}(t)-b\rangle |+\frac{\rho}{2}\|Ax_{\gamma}(t)-b\|^2\\
\leq& \mathcal{L}_{\rho}(x_{\gamma}(t),\lambda^*)-\mathcal{L}_{\rho}(x^*,\lambda^*)
+\|\lambda^*\|\|Ax_{\gamma}(t)-b\|+\frac{\rho}{2}\|Ax_{\gamma}(t)-b\|^2.
\end{split}
\end{equation*}
Then,
$
|(f+g)(x_{\gamma}(t))-(f+g)(x^*)|=\mathcal{O}\left( \frac{1}{ t \sqrt{\beta(t)}}\right), ~\textup{as} ~t\to +\infty.
$
By the properties of the Moreau-Yosida regularization reported  in \cite[Section 2]{L1997}, there exists a subsequence $\left\{ \left( x_{\gamma}(t), {\lambda}_{\gamma}(t) \right) \right\}_{\gamma >0}$ of solution of System (\ref{w1}) that converges to the
solution of System (\ref{dyn}). Thus, we obtain the desired results by passing to the limit as $\gamma\rightarrow 0$.
\qed
\end{proof}

\begin{remark}
Note that Zeng et al. \cite{zeng2023} introduced a second-order dynamical system with slow vanishing damping for solving Problem (\ref{non}). In \cite[Theorem 3.1]{zeng2023}, they obtained the $\mathcal{O}\left(\frac{1}{t^2}\right)$ convergence rate for the primal-dual gap  along the trajectory generated by the system.
Therefore, Theorem \ref{Th1} extends \cite[Theorem 3.1]{zeng2023} from dynamical systems with slow vanishing damping to those incorporating both slow vanishing damping and time scaling, thereby achieving a faster convergence rate for the primal-dual gap along the trajectory generated by System (\ref{dyn}).
\end{remark}

\section{An Inertial Accelerated Primal-dual Algorithm}
In this section, we propose an inertial accelerated primal-dual algorithm and analyze the convergence properties of this algorithm when scaling coefficient satisfies certain conditions.
First, System (\ref{dyn}) can be written as:
\begin{equation}\label{form}
\left\{
\begin{split}
&\ddot{x}(t)+\frac{\alpha}{t}\dot{x}(t)+\beta(t)\Bigg(\nabla f(x(t))+\partial g(x(t))\\
&~~~~~~~~~~~~~~+A^*\left(\lambda(t)+\frac{t}{\alpha-1} \dot{\lambda}(t)\right)+\rho A^*\left(Ax(t)-b\right)\Bigg)\ni 0,\\
&\ddot{\lambda}(t)+\frac{\alpha}{t}\dot{\lambda}(t)-\beta(t)\left(A \left(x(t)+\frac{t}{\alpha-1}\dot{x}(t) \right)-b\right)=0.
\end{split}
\right.
\end{equation}
In order to provide a reasonable time discretization of the system (\ref{form}), we follow the techniques described in \cite{a2018,hehu2022,bot2023} and let
\begin{equation*}
\left\{
\begin{split}
&u(t):=x(t)+\frac{t}{\alpha-1} \dot{x}(t),\\
&v(t):=\lambda(t)+\frac{t}{\alpha-1} \dot{\lambda}(t).
\end{split}
\right.
\end{equation*}
Then, (\ref{form}) can be reformulated as:
\begin{equation}\label{dis}
\left\{
\begin{split}
&u(t)=x(t)+\frac{t}{\alpha-1} \dot{x}(t),\\
&\dot{u}(t)\in -\frac{t}{\alpha-1} \beta(t) \left(\nabla f(x(t))+\partial g(x(t))\right)-\frac{t}{ \alpha-1 }\beta(t) A^* v(t)\\
&~~~~~~~~-\frac{t}{\alpha-1} \rho \beta(t) A^*(Ax(t)-b),\\
&v(t)=\lambda(t)+\frac{t}{\alpha-1} \dot{\lambda}(t),\\
&\dot{v}(t)=\frac{t}{\alpha-1}\beta(t)\left( A u(t)-b \right).
\end{split}
\right.
\end{equation}
For the first two lines of (\ref{dis}), we approximate $x(k+1)\approx x_{k+1}$, $u(k+1)\approx u_{k+1}$, $v(k+1)\approx v_{k+1}$, and $\beta(k)\approx \beta_k$. Applying the implicit finite-difference scheme for the first two lines of (\ref{dis}) at time $t: =k+1$ for $(x,u,v)$ and at time $t: =k$ for $\beta$, it follows that for each $k \geq 1$,
\begin{equation}\label{a4}
\left\{
\begin{split}
&u_{k+1}=x_{k+1}+\frac{k}{\alpha-1} (x_{k+1}-x_k),\\
&\widetilde{u}_{k+1}-u_k\in -\frac{k}{\alpha-1} \beta_k \left(\nabla f(\widetilde{x}_{k+1})+\partial g(x_{k+1})\right)-\frac{k}{\alpha-1}\beta_k A^* v_{k+1}\\
&~~~~~~~~~~~~~~~~~~-\frac{k}{\alpha-1} \rho \beta_k A^*(Ax_{k+1}-b),
\end{split}
\right.
\end{equation}
where $u_{k+1}$ and $\nabla f(x_{k+1})$  are replaced, respectively, by appropriate terms  $\widetilde{u}_{k+1}$ and $\nabla f(\widetilde{x}_{k+1})$ to obtain an executable iterative scheme. Similar to the setting in \cite{bot2023,ding2024}, let $\widetilde{u}_{k+1}:=u_{k+1}-\frac{\alpha-1}{k+\alpha-1}(u_{k+1}-u_k)$ and $\widetilde{x}_{k+1}:=x_k+\frac{k-1}{k+\alpha-1}(x_k-x_{k-1})$. Then, we can reformulate (\ref{a4}) as:
\begin{equation*}
\left\{
\begin{split}
&u_{k+1}=x_{k+1}+\frac{k}{\alpha-1} (x_{k+1}-x_k),\\
&u_{k+1}-u_k\in -\frac{k+\alpha-1}{\alpha-1} \beta_k \left(\nabla f(\widetilde{x}_{k+1})+\partial g(x_{k+1})\right)-\frac{k+\alpha-1}{\alpha-1}\beta_k A^* v_{k+1}\\
&~~~~~~~~~~~~~~~~-\frac{k+\alpha-1}{\alpha-1} \rho \beta_k A^*(Ax_{k+1}-b).
\end{split}
\right.
\end{equation*}

Let $\sigma>0$. For the last two lines of (\ref{dis}), we consider the time step $\sigma_k:=\sigma \left( 1+ \frac{\alpha-1}{k} \right)$ and set $\tau_k:=\sqrt{\sigma_k}k\approx \sqrt{\sigma}(k+1)$, $\lambda(\tau_k)\approx \lambda_{k+1}$,  $v(\tau_k)\approx v_{k+1}$, $u(\tau_k)\approx u_{k+1}$ and $\beta(k)\approx \beta_k$.  Evaluating the last two lines of (\ref{dis})  at time $t:=\tau_k$ for $(\lambda,v,u)$ and at time $t: =k$ for $\beta$ yields the following for each $k \geq 1$,
\begin{equation}\label{a5}
\left\{
\begin{split}
v_{k+1}=\lambda_{k+1}+\frac{\sqrt{\sigma_k}k}{\alpha-1} \frac{\lambda_{k+1}-\lambda_k}{\sqrt{\sigma_k}},\\
\frac{v_{k+1}-v_k}{\sqrt{\sigma_k}} = \frac{\sqrt{\sigma_k}k}{\alpha-1} \beta_k (Au_{k+1}-b).
\end{split}
\right.
\end{equation}
From the first equality in (\ref{a5}), we have $v_{k+1}-v_k=\frac{k+\alpha-1}{\alpha-1} \left(\lambda_{k+1}-\mu_k\right)$, where $\mu_k: =\lambda_k+\frac{k-1}{k+\alpha-1}(\lambda_k-\lambda_{k-1}) $. Furthermore, we use the following change of variables for $\{t_k\}_{k\geq 1}$ as found in \cite{attouch2018,bot2023}:
\begin{equation*}\label{r1}
t_k:=1+\frac{k-1}{\alpha-1}=\frac{k+\alpha-2}{\alpha-1},~\forall k\geq 1.
\end{equation*}
Clearly,  $t_{k+1}-1=\frac{k}{\alpha-1}$, $\frac{t_k-1}{t_{k+1}}=\frac{k-1}{k+\alpha-1}$, and $\widetilde{x}_{k+1}=x_k+\frac{t_k-1}{t_{k+1}}(x_k-x_{k-1})$. Together with (\ref{dis}), (\ref{a4}) and (\ref{a5}), we obtain the following discretization scheme of (\ref{form}):
\begin{equation}\label{aa1}
\left\{
\begin{split}
&u_{k+1}=x_{k+1}+(t_{k+1}-1) (x_{k+1}-x_k),\\
&u_{k+1}-u_k\in -t_{k+1} \beta_k \left(\nabla f(\widetilde{x}_{k+1})+\partial g(x_{k+1})\right)-t_{k+1}\beta_k A^* v_{k+1}\\
&~~~~~~~~~~~~~~~~-\rho t_{k+1} \beta_k A^*(Ax_{k+1}-b),\\
&v_{k+1}=\lambda_{k+1}+(t_{k+1}-1) (\lambda_{k+1}-\lambda_k),\\
&\mu_k=\lambda_k+\frac{t_k-1}{t_{k+1}}(\lambda_k-\lambda_{k-1}),\\
&\lambda_{k+1}=\mu_k+\sigma\beta_k (Au_{k+1}-b).
\end{split}
\right.
\end{equation}
From (\ref{aa1}), we obtain
\begin{equation*}
\begin{split}
 v_{k+1} &= t_{k+1}\lambda_{k+1}-(t_{k+1}-1) \lambda_k\\
&= t_{k+1}\mu_k+\sigma t_{k+1}\beta_k (A u_{k+1} -b ) -(t_{k+1}-1)  \lambda_k\\
&=t_{k+1}\mu_k + \sigma t_{k+1}\beta_k \left( t_{k+1} A x_{k+1}- (t_{k+1}-1) A x_k-b  \right)-(t_{k+1}-1) \lambda_k    \\
&= t_{k+1}\mu_k -(t_{k+1}-1) \lambda_k + \sigma \beta_k t_{k+1}^2 \left( Ax_{k+1}-\frac{1}{t_{k+1}}  \left((t_{k+1}-1)Ax_k +b \right) \right).
\end{split}
\end{equation*}
Let $\xi_{k+1}:=t_{k+1}\mu_k -(t_{k+1}-1)  \lambda_k$, $s_{k+1}:=\sigma \beta_k t_{k+1}^2$ and $\phi_{k+1}:=\frac{1}{t_{k+1}} \left((t_{k+1}-1)Ax_k +b \right)$. Then,
\begin{equation}\label{r3}
v_{k+1} = \xi_{k+1} + s_{k+1}  \left(Ax_{k+1}- \phi_{k+1} \right).
\end{equation}
Note that
\begin{equation*}\label{r4}
\begin{aligned}
u_{k+1}-u_k &=x_{k+1}+(t_{k+1}-1)(x_{k+1}-x_k)-x_k-(t_{k}-1)(x_k-x_{k-1})\\
&=t_{k+1}(x_{k+1}-x_k)-(t_{k}-1)(x_k-x_{k-1}).
\end{aligned}
\end{equation*}
Together with  (\ref{r3}) and  the second line of (\ref{aa1}), we have
\begin{equation*}
\begin{aligned}
0\in & (t_{k}-1)(x_k-x_{k-1})- t_{k+1}(x_{k+1}-x_k)-t_{k+1} \beta_k \left(\nabla f(\widetilde{x}_{k+1})+\partial g(x_{k+1})\right)\\
&-t_{k+1}\beta_k \left( A^* \xi_{k+1} +s_{k+1} A^* (Ax_{k+1}-\phi_{k+1}) \right)
- \rho t_{k+1} \beta_k A^* (Ax_{k+1}-b).
\end{aligned}
\end{equation*}
Then,
\begin{equation*}
\begin{aligned}
0\in & \frac{1}{ \beta_k } \left( x_{k+1}-x_k -\frac{t_{k}-1}{t_{k+1}} (x_k-x_{k-1}) \right) +  \nabla f(\widetilde{x}_{k+1})+\partial g(x_{k+1})\\
&+A^* \xi_{k+1} + \left(  s_{k+1}+ \rho \right) A^* A x_{k+1} - s_{k+1} A^* \phi_{k+1} - \rho A^* b.
\end{aligned}
\end{equation*}
This implies
\begin{equation*}
\begin{aligned}
x_{k+1}&=\mathop{\arg\min}\limits_{x\in \mathcal{H}} \Bigg\{ \langle \nabla f(\widetilde{x}_{k+1}),x \rangle+g(x)+\frac{1}{2 \beta_k }\| x-\widetilde{x}_{k+1} \|^2\\
&~~~~~~~~~~~~~~~+\frac{ \zeta_{k+1} }{2} \left\| Ax- \frac{1}{\zeta_{k+1}} \left(  s_{k+1} \phi_{k+1} + \rho b - \xi_{k+1} \right) \right\|^2 \Bigg\},
\end{aligned}
\end{equation*}
where $\zeta_{k+1}:=  s_{k+1} + \rho $.

Based on the above analysis, we are now in the position to introduce the following inertial accelerated primal-dual algorithm for solving Problem (\ref{constrained}).

\begin{algorithm}[H]
\caption{Inertial Accelerated Primal-Dual Algorithm (IAPDA)}
\textbf{Initialization:} Choose  $x_0 = x_1\in \mathcal{H} \textup{ and } \lambda_0 = \lambda_1\in \mathcal{G}$. Let $\rho> 0$, $\sigma> 0$, $\beta_0> 0$ and
  \begin{equation}\label{q0}
  \beta_{k-1} \leq \beta_k \leq  \frac{t_{k}^2}{t_{k + 1}(t_{k + 1}-1)}\beta_{k-1}, ~\forall k \geq 1.
  \end{equation}
  Let $t_1:=1$ and let $\{t_k\}_{k\geq1}$ be a nondecreasing sequence such that
  \begin{equation}\label{q1}
  t_{k + 1}^2 - t_{k + 1} - t_k^2 \leq 0, ~\forall k \geq 1.
  \end{equation}
\begin{algorithmic}
\For{$k = 1, 2, \dots$}

    \State \textbf{Step 1:} Compute
  \begin{equation}\label{q2}
  \bar{x}_k :=x_k+\frac{t_k -1}{t_{k+1}}(x_k-x_{k-1}),
  \end{equation}
  \begin{equation}\label{q4}
  \zeta_{k+1} :=  s_{k+1} + \rho,
  \end{equation}
  \begin{equation}\label{q3}
  s_{k+1} :=\sigma \beta_k t_{k+1}^2,
  \end{equation}
  \begin{equation}\label{q5}
  \phi_{k+1} :=\frac{1}{t_{k+1}} \left((t_{k+1}-1)Ax_k +b \right),
  \end{equation}
  \begin{equation}\label{q6}
  \mu_k :=\lambda_k+\frac{t_k-1}{t_{k+1}}(\lambda_k-\lambda_{k-1}),
  \end{equation}
  \begin{equation}\label{q7}
  \xi_{k+1} :=t_{k+1}\mu_k -(t_{k+1}-1)  \lambda_k.
  \end{equation}

    \State \textbf{Step 2:} Update the primal variable
  \begin{equation}\label{q8}\small
  \begin{aligned}
  x_{k+1}=&\mathop{\arg\min}\limits_{x\in \mathcal{H}} \Bigg\{ \langle \nabla f(\bar{x}_k),x \rangle+g(x)+\frac{1}{2 \beta_k }\| x-\bar{x}_k \|^2 \\
  &~~~~~~~~~~~+\frac{ \zeta_{k+1} }{2} \left\| Ax- \frac{1}{\zeta_{k+1}} \left( s_{k+1} \phi_{k+1} + \rho b - \xi_{k+1} \right) \right\|^2 \Bigg\}.
  \end{aligned}
  \end{equation}

    \State \textbf{Step 3:} Compute
  \begin{equation}\label{q9}
  u_{k+1}=x_{k+1}+(t_{k+1}-1) (x_{k+1}-x_k),
  \end{equation}
  and update the dual variable
  \begin{equation*}\label{q10}
  \lambda_{k+1}=\mu_k+\sigma\beta_k (Au_{k+1}-b).
  \end{equation*}

    \If{a stopping condition is satisfied}
        \State \Return $(x_{k+1}, \lambda_{k+1})$
    \EndIf
\EndFor
\end{algorithmic}
\end{algorithm}
\textcolor{red}{
\begin{remark}\label{remark}
It is worth noting that, in practice, it is difficult to obtain an exact solution to the subproblem
($\mathbf{Step~2 }$ in IAPDA). Accordingly, several algorithms have been developed to compute inexact solutions for such subproblems in practical problem-solving scenarios. Specifically, one can employ some classical splitting methods such as the proximal algorithm and the accelerated scheme FISTA \cite{b2020,beck2009}, as well as inexact versions derived from continuous dynamics to discrete algorithms \cite{He2025,he2026}.
\end{remark}
}
We now analyze the fast convergence rates of IAPDA. To begin, we introduce some lemmas.

\begin{lemma}\textup{\cite[Lemma 3]{hehuang2025}}\label{lemma4.1}
Suppose that $f:\mathcal{H}\rightarrow \mathbb{R}$ is  a convex function and has a Lipschitz continuous gradient with constant $L_{f}$. Then,
$$
\left\langle \nabla f(z),x-y \right\rangle \geq f(x)-f(y)-\frac{L_{f}}{2} \|x-z\|^2, ~~\forall x,y,z\in\mathcal{H}.
$$
\end{lemma}

\begin{lemma}\textup{\cite[Lemma 4]{hehu2022}}\label{lemma4.2}
Let $\{h_k\}_{k \geq 1}$ be a sequence in the space $\mathcal{H}$ and $\{a_k\}_{k \geq 1}$ be a sequence in $[0,1)$. Assume that there exists $c \geq 0$ such that
\begin{equation*}
\left\| h_{k+1}+\sum_{i=1}^k a_i h_i \right\| \leq c, ~\forall k \geq 1.
\end{equation*}
 Then,
 $$\sup_{k \geq 1} \|h_k\| \leq \|h_1\|+ 2c. $$
\end{lemma}

The following proposition will play a crucial role in the proof of the convergence rates of the sequence of iterates.

\begin{proposition}\label{theorem4.1}
Let $\{x_k,\lambda_k\}_{k \geq 0}$ be the sequence generated by $\rm{IAPDA}$ and let $\left( x^*,\lambda^* \right)\in \mathcal{S} $. Suppose that $L_f \leq \frac{1}{\beta_k}$. Then, the sequence $\{E_k\}_{k \geq 1}$ is non-increasing and
$$
\sum_{k \geq 1} ( t_{k+1}^2 \beta_k-t_{k+2}(t_{k+2}-1)\beta_{k+1}) \left( \mathcal{L}_\rho(x_{k+1},\lambda^*)-\mathcal{L}_\rho(x^*,\lambda^*) \right)<+\infty,
$$
$$
\sum_{k \geq 1} (1-L_f\beta_k) \| u_{k+1}-u_k \|^2<+\infty,
~~\sum_{k \geq 1} \left\| v_{k+1}-v_k \right\|^2<+\infty.
$$
\end{proposition}

\begin{proof}
For $\left(x^*, \lambda^* \right)\in \mathcal{S}$, we introduce the energy function which defined as
\begin{equation*}\label{bb1}
E(k)=E_0(k)+E_1(k)+E_2(k),
\end{equation*}
where
$$
E_0 (k) :=t_{k+1}(t_{k+1}-1) \beta_k \left(\mathcal{L}_{\rho}(x_k,\lambda^*)-\mathcal{L}_{\rho}(x^*,\lambda^*)\right),
$$
$$
E_1(k):=\frac{1}{2}\lVert u_{k}-x^*\rVert^2,\textup{ and } E_2(k):=\frac{1}{2\sigma}\left\lVert \xi_{k} + s_{k}  \left(Ax_{k}- \phi_{k} \right)-\lambda^* \right\rVert^2.
$$
Let $v_k:=\xi_{k} + s_{k}  \left(Ax_{k}- \phi_{k} \right)$. Then, $
E_2(k)=\frac{1}{2\sigma}\left\lVert v_k-\lambda^* \right\rVert^2.
$
From (\ref{q2})-(\ref{q8}) and the definition of $\mathcal{L}_{\rho}$, we have
\begin{equation}\label{x1}
\begin{aligned}
0&\in (t_{k}-1)(x_k-x_{k-1})- t_{k+1}(x_{k+1}-x_k)-t_{k+1} \beta_k \left(\nabla f(\bar{x}_{k})+\partial g(x_{k+1})\right)\\
&~~-t_{k+1}\beta_k A^* v_{k+1}
- \rho t_{k+1} \beta_k A^* (Ax_{k+1}-b).
\end{aligned}
\end{equation}
Let
\begin{equation}\label{g1}
\begin{aligned}
\eta_{k+1}:=& \frac{1}{t_{k+1} \beta_k} \bigg( (t_{k}-1)(x_k-x_{k-1})- t_{k+1}(x_{k+1}-x_k)-t_{k+1} \beta_k \nabla f(\bar{x}_{k}) \\
& -t_{k+1}\beta_k A^* v_{k+1}
- \rho t_{k+1} \beta_k A^* (Ax_{k+1}-b) \bigg).
\end{aligned}
\end{equation}
This together with (\ref{q9}) follows that
\begin{equation}\label{b1}
\begin{aligned}
&\frac{1}{2}\lVert u_{k+1}-x^* \rVert^2-\frac{1}{2}\lVert u_{k}-x^*\rVert^2\\
=&\left\langle u_{k+1}-u_{k}, u_{k+1}-x^* \right\rangle-\frac{1}{2} \left\| u_{k+1}-u_{k} \right\|^2\\
=& \left\langle t_{k+1}(x_{k+1}-x_k)-(t_{k}-1)(x_k-x_{k-1}),\right.\\
&\left.x_{k+1}-x^*+(t_{k+1}-1)(x_{k+1}-x_k) \right\rangle-\frac{1}{2} \left\| u_{k+1}-u_{k} \right\|^2\\
=&-t_{k+1} \beta_k\bigg\langle \nabla f(\bar{x}_k)+ \eta_{k+1} + A^* v_{k+1}+\rho A^* (Ax_{k+1}-b),\\
&x_{k+1}-x^*+(t_{k+1}-1)(x_{k+1}-x_k) \bigg\rangle-\frac{1}{2} \left\| u_{k+1}-u_{k} \right\|^2\\
=& -t_{k+1} \beta_k \left\langle \nabla f(\bar{x}_k), x_{k+1}-x^*+(t_{k+1}-1)(x_{k+1}-x_k) \right\rangle\\
&-t_{k+1}\beta_k \left\langle A^* v_{k+1}+ \eta_{k+1}, x_{k+1}-x^*+(t_{k+1}-1)(x_{k+1}-x_k) \right\rangle\\
&- t_{k+1} \beta_k \left\langle \rho A^* (Ax_{k+1}-b), x_{k+1}-x^*+(t_{k+1}-1)(x_{k+1}-x_k) \right\rangle\\
&-\frac{1}{2} \left\| u_{k+1}-u_{k} \right\|^2\\
=&-t_{k+1} \beta_k \left\langle \nabla f(\bar{x}_k), x_{k+1}-x^*+(t_{k+1}-1)(x_{k+1}-x_k) \right\rangle\\
&-t_{k+1}\beta_k \left\langle A^* (v_{k+1}- \lambda^*),  u_{k+1}-x^* \right\rangle
- t_{k+1} \beta_k \left\langle A^* \lambda^*+\eta_{k+1}, x_{k+1}-x^* \right\rangle\\
&-t_{k+1}(t_{k+1}-1) \beta_k \left\langle A^* \lambda^*+\eta_{k+1}, x_{k+1}-x_k \right\rangle\\
&- \rho t_{k+1} \beta_k \left\langle A^* (Ax_{k+1}-b), x_{k+1}-x^*+(t_{k+1}-1)(x_{k+1}-x_k) \right\rangle\\
&-\frac{1}{2} \left\| u_{k+1}-u_{k} \right\|^2,
\end{aligned}
\end{equation}
where the first equality  holds due to (\ref{o1}) and the last equality holds due to
\begin{equation*}
\begin{aligned}
&-t_{k+1}\beta_k \left\langle A^* (v_{k+1}- \lambda^*),  u_{k+1}-x^* \right\rangle\\
=&-t_{k+1}\beta_k \left\langle A^* v_{k+1}+ \eta_{k+1}- \eta_{k+1}- A^* \lambda^*, x_{k+1}-x^*+(t_{k+1}-1)(x_{k+1}-x_k) \right\rangle\\
=&- t_{k+1}\beta_k \left\langle  A^* v_{k+1}+ \eta_{k+1}, x_{k+1}-x^*+(t_{k+1}-1)(x_{k+1}-x_k) \right\rangle\\
&+ t_{k+1} \beta_k  \left\langle  A^* \lambda^*+\eta_{k+1}, x_{k+1}-x^* \right\rangle+t_{k+1}(t_{k+1}-1) \beta_k \left\langle A^* \lambda^*+ \eta_{k+1}, x_{k+1}-x_k \right\rangle.
\end{aligned}
\end{equation*}
Since $f$ is a convex function and has a Lipschitz continuous gradient with constant $L_f$, these combine with Lemma \ref{lemma4.1} to give
\begin{equation}\label{v1}
\begin{aligned}
&-t_{k+1} \beta_k \left\langle \nabla f(\bar{x}_k), x_{k+1}-x^*+(t_{k+1}-1)(x_{k+1}-x_k) \right\rangle\\
\leq & -t_{k+1}\beta_k \left( f(x_{k+1})-f(x^*) \right)-t_{k+1}(t_{k+1}-1)\beta_k \left( f(x_{k+1})-f(x_k) \right)\\
&+\frac{L_f\beta_k t_{k+1}^2}{2}\left\| x_{k+1}-\bar{x}_{k} \right\|^2\\
=& -t_{k+1}\beta_k \left( f(x_{k+1})-f(x^*) \right)-t_{k+1}(t_{k+1}-1)\beta_k \left( f(x_{k+1})-f(x_k) \right)\\
&+\frac{L_f\beta_k}{2}\left\| u_{k+1}-u_{k} \right\|^2.
\end{aligned}
\end{equation}
Moreover, it follows from (\ref{x1}) and (\ref{g1}) that $\eta_{k+1}\in \partial g(x_{k+1})$. Then,
\begin{equation}\label{b2}
\begin{aligned}
&- t_{k+1} \beta_k \left\langle A^* \lambda^*+\eta_{k+1}, x_{k+1}-x^*
 \right\rangle-t_{k+1}(t_{k+1}-1) \beta_k \left\langle A^* \lambda^*+\eta_{k+1}, x_{k+1}-x_k \right\rangle\\
\leq & -t_{k+1}\beta_k \left( g(x_{k+1})-g(x^*) \right)-t_{k+1}(t_{k+1}-1)\beta_k \left( g(x_{k+1})-g(x_k) \right)\\
&- t_{k+1}\beta_k \left\langle  A^* \lambda^*, x_{k+1}-x^*+(t_{k+1}-1)(x_{k+1}-x_k) \right\rangle.
\end{aligned}
\end{equation}
Note that
\begin{equation*}
\begin{aligned}
&-\rho t_{k+1} \beta_k \left\langle A^*(Ax_{k+1}-b), (x_{k+1}-x^*)+(t_{k+1}-1)(x_{k+1}-x_k) \right\rangle\\
=&- \rho t_{k+1} \beta_k \left\| Ax_{k+1}-b \right\|^2
- \rho t_{k+1}(t_{k+1}-1) \beta_k \left\langle Ax_{k+1}-b, A x_{k+1}-A x_k \right\rangle\\
=&- \rho t_{k+1} \beta_k \left\| Ax_{k+1}-b \right\|^2 +\frac{\rho}{2} t_{k+1}(t_{k+1}-1) \beta_k \left\| Ax_{k}-b \right\|^2\\
& -\frac{\rho}{2} t_{k+1}(t_{k+1}-1) \beta_k \left\| Ax_{k+1}-b \right\|^2-\frac{\rho}{2}  t_{k+1}(t_{k+1}-1) \beta_k \left\| Ax_{k+1}-Ax_k \right\|^2\\
\leq & - \rho t_{k+1} \beta_k \left\| Ax_{k+1}-b \right\|^2 +\frac{\rho}{2} t_{k+1}(t_{k+1}-1) \beta_k \left\| Ax_{k}-b \right\|^2\\
&-\frac{\rho}{2} t_{k+1}(t_{k+1}-1) \beta_k \left\| Ax_{k+1}-b \right\|^2.
\end{aligned}
\end{equation*}
Together with (\ref{b1}), (\ref{v1}) and (\ref{b2}), we have
\begin{equation*}\label{b5}
\begin{aligned}
&E_1(k+1)-E_1(k)\\
=&\frac{1}{2}\lVert u_{k+1}-x^* \rVert^2-\frac{1}{2}\lVert u_{k}-x^*\rVert^2\\
\leq& -t_{k+1}\beta_k \left\langle A^* (v_{k+1}- \lambda^*),  u_{k+1}-x^* \right\rangle -t_{k+1}\beta_k \Big( f(x_{k+1})+g(x_{k+1})-f(x^*)-g(x^*) \Big)\\
&-t_{k+1}(t_{k+1}-1)\beta_k \Big( f(x_{k+1})+g(x_{k+1})-f(x_k)-g(x_{k}) \Big)\\
&- t_{k+1}\beta_k \left\langle  A^* \lambda^*, x_{k+1}-x^*+(t_{k+1}-1)(x_{k+1}-x_k) \right\rangle- \rho t_{k+1} \beta_k \left\| Ax_{k+1}-b \right\|^2\\
&+\frac{\rho}{2} t_{k+1}(t_{k+1}-1) \beta_k \left\| Ax_{k}-b \right\|^2-\frac{\rho}{2} t_{k+1}(t_{k+1}-1) \beta_k \left\| Ax_{k+1}-b \right\|^2\\
&+ \frac{L_f\beta_k-1}{2}\left\| u_{k+1}-u_{k} \right\|^2\\
\leq &-t_{k+1}\beta_k \left\langle A^* (v_{k+1}- \lambda^*),  u_{k+1}-x^* \right\rangle
- t_{k+1} \beta_k \left( \mathcal{L}_{\rho} (x_{k+1}, \lambda^*)-\mathcal{L}_{\rho} (x^*, \lambda^*) \right)\\
&-t_{k+1}(t_{k+1}-1) \beta_k \left( \mathcal{L}_{\rho} (x_{k+1}, \lambda^*)-\mathcal{L}_{\rho} (x_k, \lambda^*) \right)+\frac{L_f\beta_k-1}{2}\left\| u_{k+1}-u_{k} \right\|^2.
\end{aligned}
\end{equation*}
On the other hand,
\begin{equation*}
\begin{aligned}
&E_2(k+1)-E_2(k)\\
=&\frac{1}{2 \sigma}\lVert v_{k+1}- \lambda^* \rVert^2-\frac{1}{2 \sigma}\lVert v_{k}- \lambda^*\rVert^2\\
=&\frac{1}{\sigma} \left\langle v_{k+1}-v_{k},  v_{k+1}- \lambda^* \right\rangle
-\frac{1}{2 \sigma}\lVert v_{k+1}- v_{k}  \rVert^2\\
=&\frac{1}{\sigma} \left\langle t_{k+1} (\lambda_{k+1}-\mu_k),  v_{k+1}- \lambda^* \right\rangle -\frac{1}{2 \sigma}\lVert v_{k+1}- v_{k}  \rVert^2\\
=&t_{k+1} \beta_k \left\langle A u_{k+1}-b ,  v_{k+1}- \lambda^* \right\rangle-\frac{1}{2 \sigma}\lVert v_{k+1}- v_{k}  \rVert^2.
\end{aligned}
\end{equation*}
Thus,
\begin{equation}\label{b6}
\begin{split}
&E(k+1)-E(k)\\
=&E_0(k+1)-E_0(k)+E_1(k+1)-E_1(k)+E_2(k+1)-E_2(k)\\
\leq& \left( t_{k+2}(t_{k+2}-1)\beta_{k+1} - t_{k+1}^2 \beta_k \right) \left(\mathcal{L}_{\rho}(x_{k+1},\lambda^*)-\mathcal{L}_{\rho}(x^*,\lambda^*)\right)\\
&+\frac{L_f\beta_k-1}{2}\left\| u_{k+1}-u_{k} \right\|^2-\frac{1}{2 \sigma}\lVert v_{k+1}- v_{k}  \rVert^2.
\end{split}
\end{equation}
From (\ref{q0}), we have
\begin{equation*}
\begin{aligned}
t_{k+2}(t_{k+2}-1)\beta_{k+1} -t_{k+1}^2 \beta_k\leq 0.
\end{aligned}
\end{equation*}
By $L_f \leq \frac{1}{\beta_k}$, we have $L_f\beta_k-1 \leq 0$. Then, all the coefficients in the right-hand side of (\ref{b6}) are non-positive. Thus, the sequence $\{ E(k) \}_{k\geq 1}$ is non-increasing.

Summing (\ref{b6}) over $k=1,\dots,N$, we obtain
\begin{equation*}
\begin{aligned}
& \sum_{k=1}^N ( t_{k+1}^2 \beta_k-t_{k+2}(t_{k+2}-1)\beta_{k+1}) \left(\mathcal{L}_{\rho}(x_{k+1},\lambda^*)-\mathcal{L}_{\rho}(x^*,\lambda^*)\right)\\
&~~+ \sum_{k=1}^N \frac{1-L_f\beta_k}{2} \left\| u_{k+1}-u_{k} \right\|^2+\frac{1}{2 \sigma} \sum_{k=1}^N \lVert v_{k+1}- v_{k}  \rVert^2
\\
&\leq E(1)-E(N+1)\leq E(1).
\end{aligned}
\end{equation*}
Letting $N\rightarrow +\infty$ in the above inequality, we obtain
$$
\sum_{k \geq 1} ( t_{k+1}^2 \beta_k-t_{k+2}(t_{k+2}-1)\beta_{k+1}) \left( \mathcal{L}_\rho(x_{k+1},\lambda^*)-\mathcal{L}_\rho(x^*,\lambda^*) \right)<+\infty,
$$
$$
\sum_{k \geq 1} (1-L_f\beta_k) \| u_{k+1}-u_k \|^2<+\infty,
~~\sum_{k \geq 1} \left\| v_{k+1}-v_k \right\|^2<+\infty.
$$
The proof is complete.\qed
\end{proof}

Next, by following a similar argument as in \cite[Proposition 3.13]{bot2023},  we obtain the boundedness of the sequence of the primal-dual iterates $\{(x_k, \lambda_k)\}_{k\geq 0}$ in the following proposition.

\begin{proposition}
Let $\{(x_k, \lambda_k)\}_{k\geq 0}$ be the sequence generated by $\mathrm{IAPDA}$. Suppose that
\begin{equation}\label{b7}
\tau:=\inf_{k\geq 1}\frac{t_k}{k}>0.
\end{equation}
Then, the sequences $\{(x_k, \lambda_k)\}_{k\geq 0}$ and $\{t_{k+1}(x_{k+1}-x_{k},\lambda_{k+1}-\lambda_{k})\}_{k\geq 0}$  are bounded.
\end{proposition}

Now, we prove the convergence rates of the trajectory generated by IAPDA.

\begin{theorem}\label{theorem}
Let $\{(x_k, \lambda_k)\}_{k\geq 0}$ be the sequence generated by $\mathrm{IAPDA}$ and let $\left( x^*,\lambda^* \right)\in \mathcal{S} $.
Then,
$$
\|x_{k+1}-x_k\|= \mathcal{O}\left( \frac{1}{k} \right)
\textup{ and }
\|\lambda_{k+1}-\lambda_k\|= \mathcal{O}\left( \frac{1}{k} \right), \textup{ as } k\rightarrow +\infty.
$$
\end{theorem}

\begin{proof}
From (\ref{q9}), we have
$$
u_{k+1}-u_k=x_k-x_{k-1}+t_{k+1}(x_{k+1}-x_{k})-t_k(x_k-x_{k-1}).
$$
Considering that $x_0=x_1$, we obtain
\begin{equation*}
u_{k+1}-u_1 = \sum_{i=1}^k (u_{i+1}-u_i) = t_{k+1}(x_{k+1}-x_{k})+\sum_{i=1}^k (x_{i}-x_{i-1}).
\end{equation*}
Since $\{u_k\}_{k \geq 1}$ is bounded, there exists an $M>0$ such that
\begin{equation*}
\left\| t_{k+1}(x_{k+1}-x_{k})+\sum_{i=1}^k (x_{i}-x_{i-1}) \right\| \leq M, ~\forall k \geq 1.
\end{equation*}
This together with Lemma \ref{lemma4.2} gives
\begin{equation*}
\sup_{k\geq 1} t_k \|x_k-x_{k-1}\| \leq 2M.
\end{equation*}
Thus,
\begin{equation*}
\|x_k-x_{k-1}\| = \mathcal{O} \left(\frac{1}{k} \right), \textup{ as } k\rightarrow +\infty.
\end{equation*}
Similarly, by using the boundedness of $\{v_k\}_{k \geq 1}$ and Lemma \ref{lemma4.2}, we deduce that
\begin{equation*}
\left\|\lambda_k-\lambda_{k-1}\right\| = \mathcal{O} \left(\frac{1}{k} \right), \textup{ as } k\rightarrow +\infty.
\end{equation*}
\end{proof}

Next, we investigate the fast convergence rates for the  primal-dual gap, the feasibility violation, and the objective residual.

\begin{theorem}\label{theorem4.2}
Let $\{(x_k, \lambda_k)\}_{k\geq 0}$ be the sequence generated by $\mathrm{IAPDA}$ and let $\left( x^*, \lambda^* \right) \in \mathcal{S} $. Then, for each $k \geq 1$, it holds that
$$
\mathcal{L}_{\rho}(x_k, \lambda^*)-\mathcal{L}_{\rho}(x^*, \lambda^*)\leq \frac{E(1)}{ t_{k+1}(t_{k+1}-1) \beta_k},$$
$$\| A x_k -b \| \leq \frac{\beta_0 t_1^2\| A x_1-b \|+2C }{\beta_k t_{k+1} (t_{k+1}-1)},
$$
and
\begin{equation*}
\begin{aligned}
\left | (f+g)(x_k) -(f+g)(x^*) \right |
\leq& \frac{E(1)}{\beta_k t_{k+1} (t_{k+1}-1) } + \left\| \lambda^* \right\| \frac{\beta_0 t_1^2\| A x_1-b \|+2C }{\beta_k t_{k+1} (t_{k+1}-1)}\\
&+\frac{\rho}{2} \left( \frac{\beta_0 t_1^2\| A x_1-b \|+2C }{\beta_k t_{k+1} (t_{k+1}-1)} \right)^2,
\end{aligned}
\end{equation*}
where $ C:= \| m_1 \|+ \frac{1}{\sigma} \sup\limits_{k \geq 1}\| t_{k+1} (\lambda_{k+1}-\lambda_k) \|+ \frac{1}{\sigma} t_1 \| \lambda_1-\lambda_0 \|+ \frac{1}{\sigma} \sup\limits_{k \geq 1}\| \lambda_k \|+ \frac{1}{\sigma} \| \lambda_0\| $, $ m_1:= t_1^2 \beta_{0} (A x_1-b) $, and
\begin{equation*}
\begin{aligned}
E(1)&=t_{2}(t_{2}-1) \beta_1 \left(\mathcal{L}_{\rho}(x_1,\lambda^*)-\mathcal{L}_{\rho}(x^*,\lambda^*)\right)+\frac{1}{2}\lVert u_{1}- x^* \rVert^2+\frac{1}{2\sigma}\lVert v_{1}- \lambda^* \rVert^2.
\end{aligned}
\end{equation*}
\end{theorem}

\begin{proof}
From the proof of Proposition \ref{theorem4.1}, we have
$$ t_{k+1}(t_{k+1}-1) \beta_k \left( \mathcal{L}_\rho(x_{k},\lambda^*)-\mathcal{L}_\rho(x^*,\lambda^*) \right) \leq E(k) \leq E(1).
$$
Thus,
$$ \mathcal{L}_\rho(x_{k},\lambda^*)-\mathcal{L}_\rho(x^*,\lambda^*)  \leq \frac{E(1)}{ t_{k+1}(t_{k+1}-1) \beta_k}.
$$
Since $\{t_{k+1}(\lambda_{k+1}-\lambda_{k})\}_{k\geq 0}$ and $\{ \lambda_k \}_{k\geq 0}$ are bounded and
$$
t_{k+1} (\lambda_{k+1}-\mu_k)=t_{k+1}(\lambda_{k+1}-\lambda_k)-(t_k-1)(\lambda_k-\lambda_{k-1}),~\forall k\geq 1,
$$
it follows that $\{ t_{k+1} (\lambda_{k+1}-\mu_k) \}_{k\geq 1}$ is bounded.
On the other hand,
\begin{equation}\label{z1}
\begin{aligned}
t_{k+1} (\lambda_{k+1}-\mu_k )
=&\sigma t_{k+1} \beta_k \left( A u_{k+1}- b \right)\\
=& \sigma t_{k+1} \beta_k \left( A x_{k+1}-b+ (t_{k+1}-1)(A x_{k+1}-A x_{k}) \right)\\
=&\sigma t_{k+1} \beta_k \left( t_{k+1} (A x_{k+1}-b) -(t_{k+1}-1)(A x_{k}-b) \right)\\
=&\sigma \left( m_{k+1}+ (a_k-1) m_k \right),
\end{aligned}
\end{equation}
where $ m_k:= t_k^2 \beta_{k-1} (A x_k-b) $ and $ a_k:=1- \frac{ t_{k+1} (t_{k+1}-1) \beta_k }{ t_k^2 \beta_{k-1} } $. Note that
$$
\sum_{i=1}^k (m_{i+1}+ (a_i-1) m_i )=m_{k+1}-m_1 + \sum_{i=1}^k a_i m_i.
$$
This together with (\ref{z1}) gives
\begin{equation*}
\begin{aligned}
m_{k+1}+\sum_{i=1}^k a_i m_i
=& m_1 + \frac{1}{\sigma} \sum_{i=1}^k t_{i+1} (\lambda_{i+1}-\mu_i )\\
=& m_1 + \frac{1}{\sigma} \sum_{i=1}^k \left( t_{i+1} (\lambda_{i+1}-\lambda_i )- (t_i-1) (\lambda_i- \lambda_{i-1} ) \right) \\
=& m_1+ \frac{1}{\sigma} \left( t_{k+1} (\lambda_{k+1}-\lambda_k )-t_1 (\lambda_1- \lambda_0 )+ (\lambda_k- \lambda_{0} ) \right).
\end{aligned}
\end{equation*}
Thus,
$$
\left\| m_{k+1}+ \sum_{i=1}^k a_i m_i \right \|= \left\| m_1+ \frac{1}{\sigma} \left( t_{k+1} (\lambda_{k+1}-\lambda_k )-t_1 (\lambda_1- \lambda_0 )+ (\lambda_k- \lambda_{0} ) \right) \right\| \leq C,
$$
where $C:= \| m_1 \|+ \frac{1}{\sigma} \sup\limits_{k \geq 1}\| t_{k+1} (\lambda_{k+1}-\lambda_k) \|+ \frac{1}{\sigma} t_1 \| \lambda_1-\lambda_0 \|+ \frac{1}{\sigma} \sup\limits_{k \geq 1}\| \lambda_k \|+ \frac{1}{\sigma} \| \lambda_0\| $. Moreover,  from (\ref{q0}), we have $0 \leq a_k <1 $, $\forall k \geq 1$. Then,
$$
\| A x_k -b \| \leq \frac{\beta_0 t_1^2\| A x_1-b \|+2C }{\beta_{k-1} t_k^2} \leq \frac{\beta_0 t_1^2\| A x_1-b \|+2C }{\beta_k t_{k+1} (t_{k+1}-1)},
$$
where the first inequality holds due to Lemma \ref{lemma4.2} and  the second inequality holds  due to (\ref{q0}). This together with the definition of $\mathcal{L}_{\rho}$  yields
\begin{equation*}
\begin{aligned}
&\left | (f+g)(x_k) -(f+g)(x^*) \right | \\
\leq&  \mathcal{L}_{\rho}(x_k, \lambda^*)-\mathcal{L}_{\rho}(x^*, \lambda^*) + \left\| \lambda^* \right\| \|Ax_k-b\|
+\frac{\rho}{2} \|Ax_k-b\|^2\\
\leq& \frac{E(1)}{ \beta_k t_{k+1} (t_{k+1}-1) } + \left\| \lambda^* \right\| \frac{\beta_0 t_1^2\| A x_1-b \|+2C }{\beta_k t_{k+1} (t_{k+1}-1)}\\
&+\frac{\rho}{2} \left( \frac{\beta_0 t_1^2 \| A x_1-b \|+2C }{\beta_k t_{k+1} (t_{k+1}-1)} \right)^2.
\end{aligned}
\end{equation*}
The proof is complete.\qed
\end{proof}

In the following,  in terms of using Nesterov's rule \cite{n1983}, Chambolle-Dossal rule \cite{c2015} and Attouch-Cabot rule \cite{attouch2018}, we can obtain the convergence rates for the primal-dual gap, the feasibility violation, and the objective residual introduced in Theorem \ref{theorem4.2}.
\begin{remark}
(i) We consider the Nesterov's rule as proposed in \cite{n1983}:
$$
t_1:=1 \textup{ and } t_{k+1}:=\frac{1+\sqrt{1+4 t_k^2 } }{2},~\forall k\geq 1.
$$
Clearly, this sequence fulfills (\ref{q1}). On the other hand, we have (see, for instance, \cite[Lemma 4.3]{beck2009}), $t_k \geq \frac{k+1}{2}$, $\forall k \geq 1$. Thus, (\ref{b7}) is verified for $\tau \geq \frac{1}{2}$. Furthermore, since $t_{k+1}(t_{k+1}-1) \geq \frac{k(k+2)}{ 2 } \geq \frac{k^2}{2}$, $\forall k \geq 1$, it follows that
$$
\mathcal{L}_{\rho}(x_k, \lambda^*)-\mathcal{L}_{\rho}(x^*, \lambda^*)=\mathcal{O}\left( \frac{1}{k^2  \beta_{k}} \right), \textup{ as } k\rightarrow +\infty,
$$
$$
\| A x_k -b \| = \mathcal{O} \left( \frac{1}{ k^2 \beta_k } \right) \textup{ and }
| (f+g)(x_k) -(f+g)(x^*) | = \mathcal{O} \left( \frac{1}{ k^2 \beta_k } \right), \textup{ as } k\rightarrow +\infty.
$$
\\
(ii) We consider the Chambolle-Dossal rule as proposed in \cite{c2015}:
$$
t_k:=1+\frac{k-1}{\alpha-1}=\frac{k+\alpha-2}{\alpha-1},~\forall k \geq 1, ~ \alpha \geq 3.
$$
We first show that this sequence fulfills (\ref{q1}). Indeed,
$$
t_{k + 1}^2 - t_{k + 1} - t_k^2 = \frac{k(3-\alpha)-(\alpha-2)^2}{ (\alpha-1)^2 } < 0, ~\forall k \geq 1.
$$
On the other hand, for each $k \geq 1$,
$$
\frac{t_k}{k}=\frac{1}{\alpha-1}+\frac{\alpha-2}{k(\alpha-1)},
$$
which implies that (\ref{b7}) is verified for $\tau= \frac{1}{ \alpha-1}$. Furthermore, since $t_{k+1}(t_{k+1}-1)=\frac{k}{\alpha-1}\left( 1+\frac{k}{\alpha-1} \right) \geq \frac{k^2}{ (\alpha-1)^2 }$, $\forall k \geq 1$, it follows that
$$
\mathcal{L}_{\rho}(x_k, \lambda^*)-\mathcal{L}_{\rho}(x^*, \lambda^*)=\mathcal{O}\left( \frac{1}{k^2  \beta_{k}} \right), \textup{ as } k\rightarrow +\infty,
$$
$$
\| A x_k -b \| = \mathcal{O} \left( \frac{1}{ k^2 \beta_k } \right) \textup{ and }
| (f+g)(x_k) -(f+g)(x^*) | = \mathcal{O} \left( \frac{1}{ k^2 \beta_k } \right), \textup{ as } k\rightarrow +\infty.
$$
(iii) We consider the Attouch-Cabot rule as proposed in \cite{attouch2018}:
$$
t_k:=\frac{k-1}{\alpha-1},~\forall k \geq 1,~\alpha \geq 3.
$$
We first show that this sequence fulfills (\ref{q1}). Indeed,
$$
t_{k + 1}^2 - t_{k + 1} - t_k^2 = \frac{k(3-\alpha)-1}{ (\alpha-1)^2 } < 0, ~\forall k \geq 1.
$$
Since $t_k\geq 1$ in IAPDA, we have $k \geq k_1:= [\alpha]+1$. Then, for each $k \geq k_1$,
$$
\frac{t_k}{k}=\frac{1}{\alpha-1}-\frac{1}{k(\alpha-1)},
$$
which implies that (\ref{b7}) is verified for $\tau= \frac{3}{4(\alpha-1)}$. Furthermore, since
$$
t_{k+1}(t_{k+1}-1)=k^2\left( \frac{1}{(\alpha-1)^2}-\frac{1}{k(\alpha-1)} \right) =\mathcal{O}(k^2), ~\mbox{as}~ k\rightarrow +\infty,
$$
it follows that
$$
\mathcal{L}_{\rho}(x_k, \lambda^*)-\mathcal{L}_{\rho}(x^*, \lambda^*)=\mathcal{O}\left( \frac{1}{k^2  \beta_{k}} \right), \textup{ as } k\rightarrow +\infty,
$$
$$
\| A x_k -b \| = \mathcal{O} \left( \frac{1}{ k^2 \beta_k } \right) \textup{ and }
| (f+g)(x_k) -(f+g)(x^*) | = \mathcal{O} \left( \frac{1}{ k^2 \beta_k } \right), \textup{ as } k\rightarrow +\infty.
$$
\end{remark}

\begin{remark}
From the time discretization of a second-order dynamical system with slow vanishing damping, Bo\c{t} et al. \cite{bot2023} proposed a fast augmented Lagrangian algorithm  for solving Problem (\ref{non}). Then, in \cite[Theorem 3.17]{bot2023} and \cite[Theorem 3.18]{bot2023}, they obtained the $\mathcal{O}\left( \frac{1}{ k^2 }\right)$ convergence rate for the primal-dual gap, the feasibility violation, and the objective residual. Thus, Theorem \ref{theorem} can be viewed as a generalization of \cite[Theorem 3.17]{bot2023} and \cite[Theorem 3.18]{bot2023}.
\end{remark}

\section{Numerical experiments}
In this section,  we validate the theoretical findings of previous sections through numerical experiments. In these experiments, all codes are implemented in MATLAB R2021b and executed on a PC equipped with a 2.40GHz Intel Core i5-1135G7 processor and 16GB  of RAM.

\begin{example}\label{example5.2}
Consider the following $ \ell_1-\ell_2 $ minimization problem
\begin{eqnarray}\label{exp5.2}
\left\{ \begin{split}
&\mathop{\mbox{min}}\limits_{x\in \mathbb{R}^n} \|x\|_1+\frac{\mu}{2} \|x\|^2\\
&\mbox{s.t.}~~Ax=b,
\end{split}
\right.
\end{eqnarray}
where $\mu\geq 0$, $ A\in \mathbb{R}^{m\times n} $ and
$ b\in\mathbb{R}^m $. We note that Problem (\ref{exp5.2}) has been used in a wide range of fields, such as  signal processing, machine learning, and statistics, especially in sparse signal recovery and feature selection problems. See \cite{a2018,f2011} and the references therein.
\vspace{-1em}
 \begin{figure}[H]%
     \centering
         \subfloat[$  \gamma=10^{-4} $ ]{
         \includegraphics[width=0.48\linewidth]{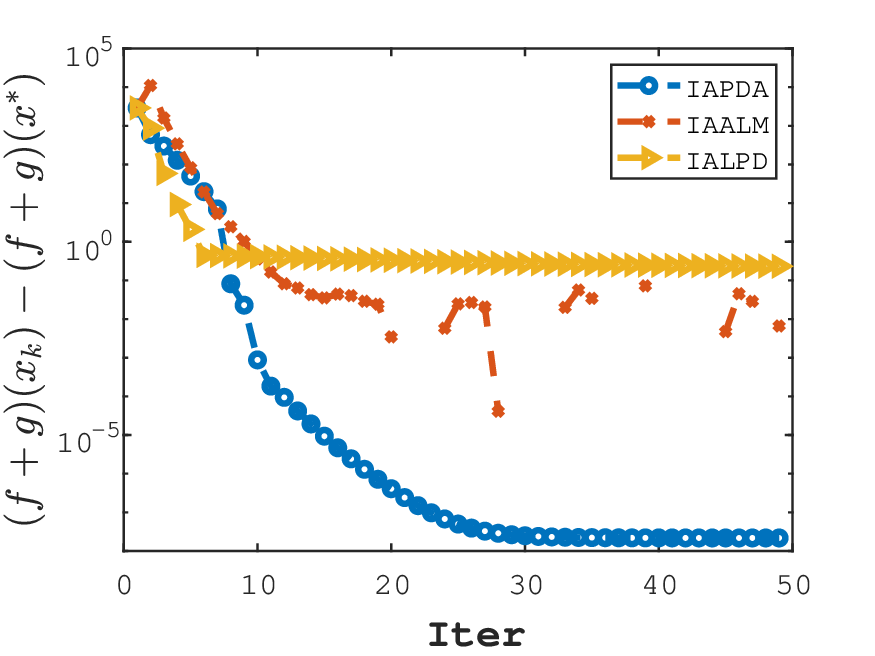}
         \hfill
         \includegraphics[width=0.48\linewidth]{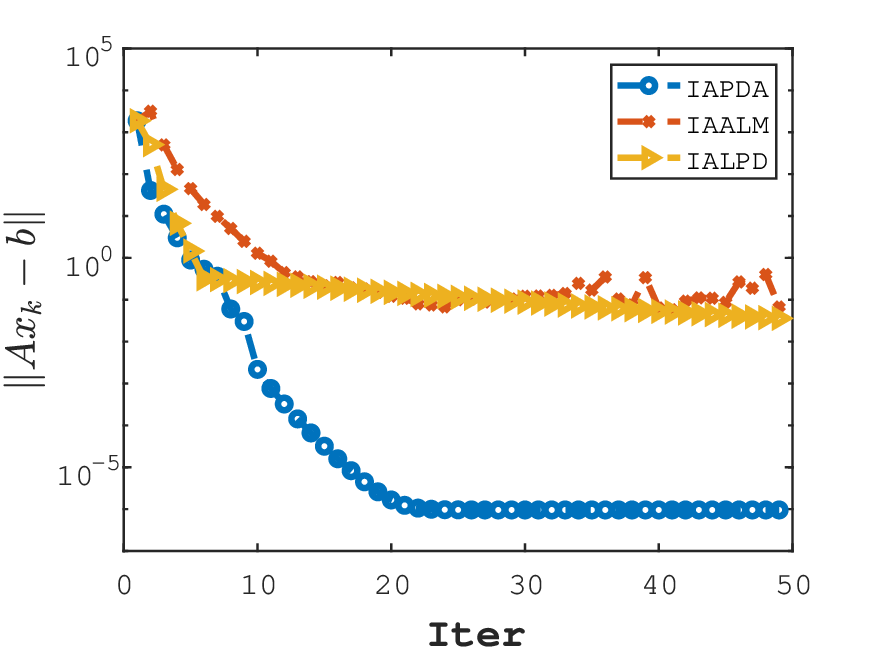}
         }\\
     \subfloat[$  \gamma=10^{-6} $ ]{
         \includegraphics[width=0.48\linewidth]{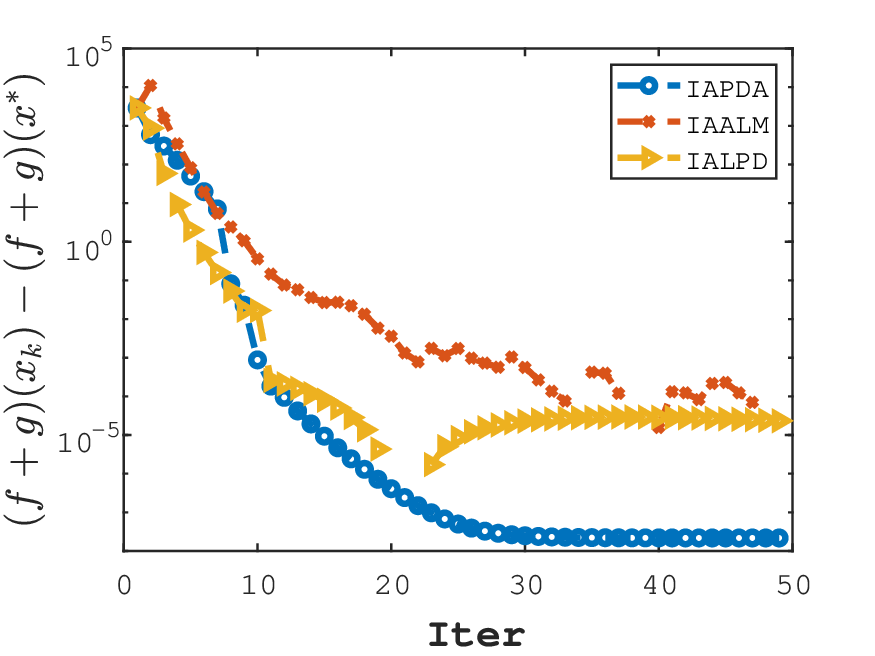}
         \hfill
         \includegraphics[width=0.48\linewidth]{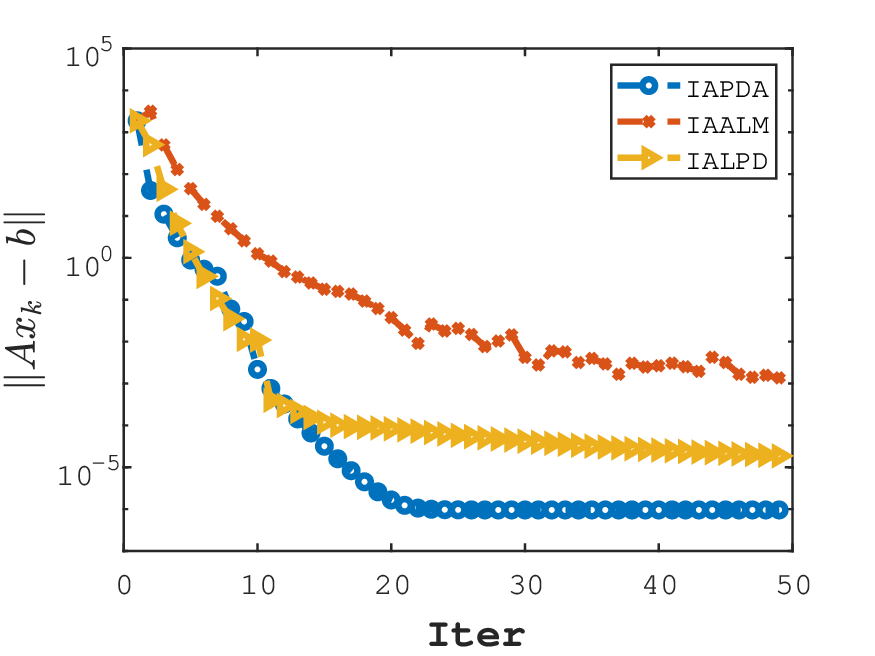}
         }\\
     \subfloat[$  \gamma=10^{-8} $ ]{
         \includegraphics[width=0.48\linewidth]{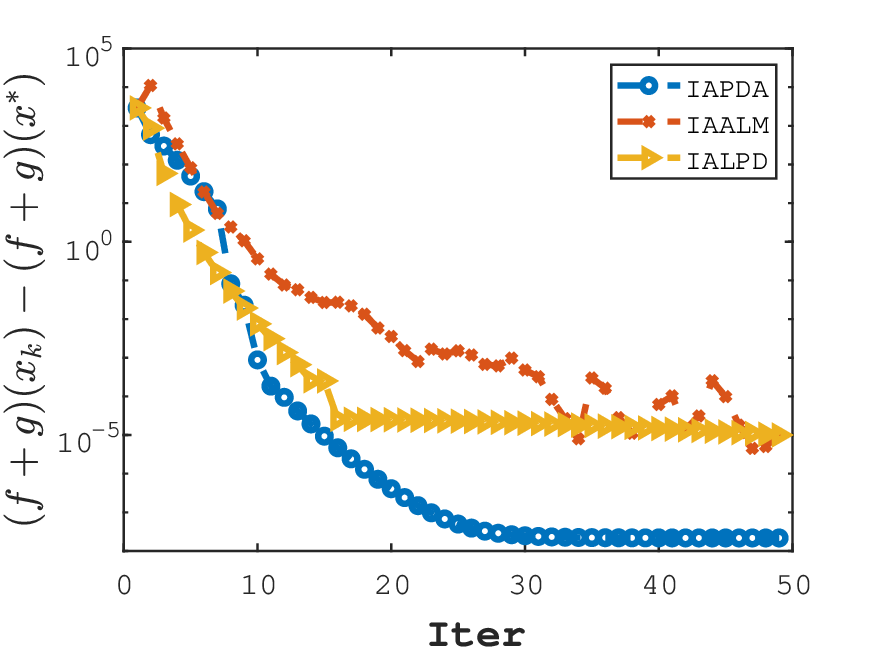}
         \hfill
         \includegraphics[width=0.48\linewidth]{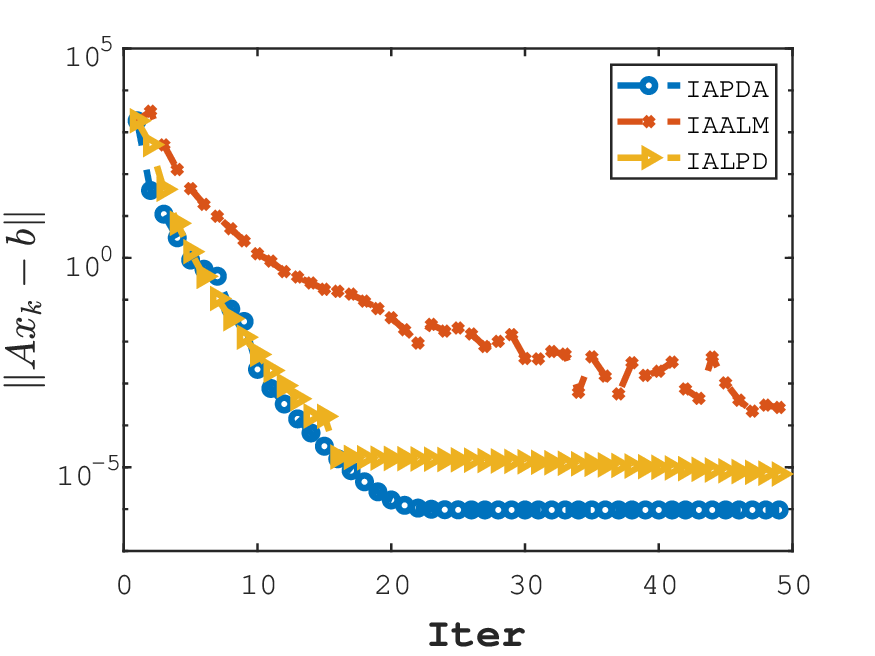}
         }\\
     \caption{Numerical results of IAPDA, IAALM and IALPD under various subtolerances.}
     \label{HD-exp2}
 \end{figure}
 \vspace{-1em}

Let $m=1500$, $n=2000$ and $\mu=1.5$. We generate the matrix $A$ using the standard Gaussian distribution. The original solution $x^*\in\mathbb{R}^n$ is generated from a Gaussian distribution $\mathcal{N}(0,4)$, with its entries clipped to the interval [-2,2] and sparsified so that only $5\%$ of its elements are non-zero. The noise $\omega$ is generated from a standard Gaussian distribution and normalized
to have a norm of $\|\omega\| = 10^{-6}$. Furthermore, we select
$
b=Ax^*+\omega.
$

We solve the subproblem  occurring in IAPDA using the fast iterative shrinkage-thresholding algorithm (FISTA, \cite{beck2009}). The stopping condition is
$$
\frac{ \left\|z^{\prime}_k-z^{\prime}_{k-1} \right\| }{ \max \{ \left\|z_{k-1}^{\prime} \right\|,1\} } \leq \gamma
$$
or the number of iterations exceeds 150. Here, the $z^{\prime}_k$ are the iterates generated by FISTA with the accuracy $\gamma:=\{10^{-4}$, $10^{-6}$, $10^{-8}\}$, respectively.
\textcolor{red}{As noted in Remark \ref{remark}, the subproblem generally does not admit an explicit closed-form solution. In this example, the subproblem is solved approximately by the FISTA with a sufficiently small stopping tolerance. Although the theoretical convergence of IAPDA is established under the assumption that the subproblem is solved exactly, such a highly accurate approximate solution can be regarded as a small perturbation of the exact one. As shown similarly in \cite{s2011,He2025,he2026}, this mild approximation error does not affect the convergence properties under proper error control.}

In the following experiments, we solve Problem (\ref{exp5.2}) on the time interval $[1, 50]$. Set $t_0=t_1:=1$.
We compare IAPDA with the inertial accelerated augmented Lagrangian method (IAALM) proposed in \cite{kang2015} and the inertial accelerated linearized primal-dual method (IALPD) proposed in \cite{he2022}.  Here are the algorithm parameter settings:
\begin{itemize}
\item IAPDA:  $ \rho=10^{-4}$, $\sigma=10$, $\beta_0=2$ and $\alpha=15$.
\item IAALM: $\tau=0.01$.
\item IALPD: $s=1$, $M_k=\frac{s I_n}{n}$ and $\alpha=15$.
\end{itemize}

As shown in Figure \ref{HD-exp2}, which displays numerical results for various $\gamma$ and the first $100$ iterations, IAPDA achieves significantly superior performance over IAALM and IALPD under various subtolerances.
\end{example}

\begin{example}\label{example5.3}
Consider the following non-negative least squares problem
\begin{eqnarray}\label{exp5.3}
\mathop{\mbox{min}}\limits_{x\in \mathbb{R}^n} f(x):=\frac{1}{2} \|Ax-b\|^2,
\end{eqnarray}
where $ A\in \mathbb{R}^{m\times n} $ and $ b\in\mathbb{R}^m $. We generate a random matrix $A\in \mathbb{R}^{m \times n}$ with density $\gamma\in (0,1]$ and a random $ b\in\mathbb{R}^m $. The nonzero entries of $A$ are independently generated from a uniform distribution in $[0,0.1]$.

In the following experiments, we solve Problem (\ref{exp5.3}) on the time interval $[1, 2000]$. Set $ t_0=t_1:=1 $. We compare IAPDA with FISTA proposed in \cite{beck2009} and the accelerated forward-backward method (AFBM) proposed in \cite{a2016}. Here are the algorithm parameter settings:
\begin{itemize}
\item IAPDA: $\beta_0=1$.
\item FISTA: $s=\frac{1}{\|A\|^2}$.
\item AFBM: $s=\frac{1}{\|A\|^2}$ and $\alpha=5$.
\end{itemize}
Moreover, for $\gamma=0.5$ and $\gamma=1$, two different dimension settings are respectively considered:
\begin{itemize}
\item $  m=500$,   $n=1000$.
\item $  m=1500$,   $n=2000$.
\end{itemize}

As shown in Figures \ref{HD-exp3} and \ref{HD-exp4}, compared to FISTA and AFBM, IAPDA can achieve higher accuracy on two different dimension settings. Moreover, IAPDA achieves more stable performance, while the convergence speed and accuracy of FISTA and AFBM
are obviously affected by $\gamma$.

\vspace{-1em}
 \begin{figure}[H]%
     \centering
         \subfloat[$  \gamma=0.5 $ ]{
         \includegraphics[width=0.48\linewidth]{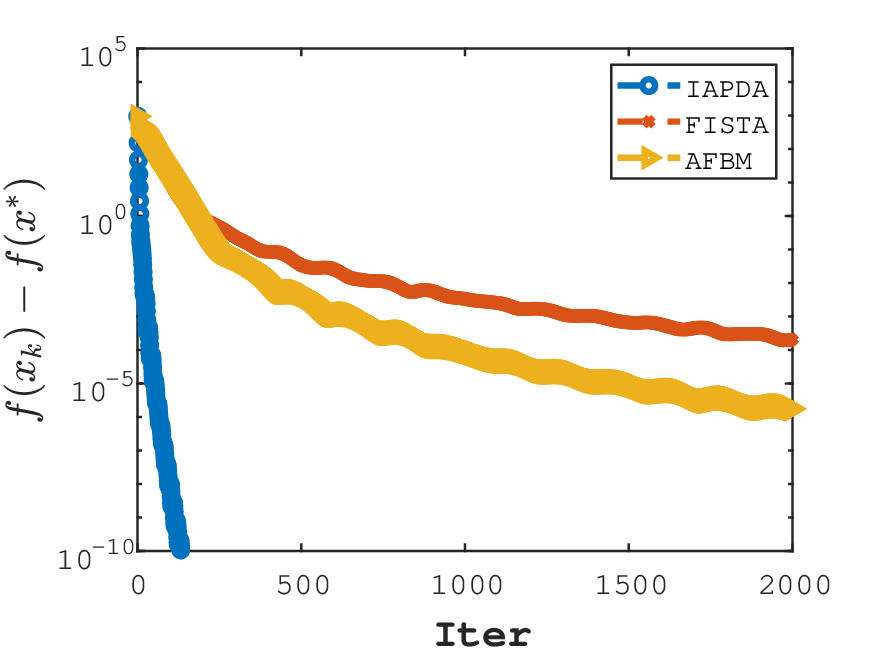}
         }
         \subfloat[$  \gamma=1 $ ]{
         \includegraphics[width=0.48\linewidth]{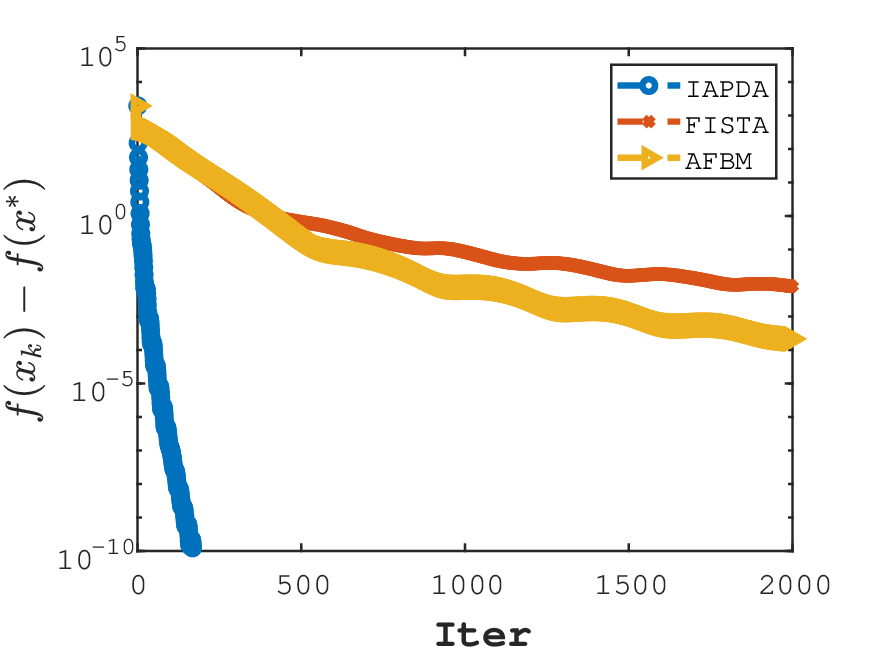}
         }\\
     \caption{Numerical results of IAPDA, FISTA and AFBM when $m=500$, $n=1000$.}
     \label{HD-exp3}
 \end{figure}
 \vspace{-1em}

\vspace{-1em}
 \begin{figure}[H]%
     \centering
         \subfloat[$  \gamma=0.5 $ ]{
         \includegraphics[width=0.48\linewidth]{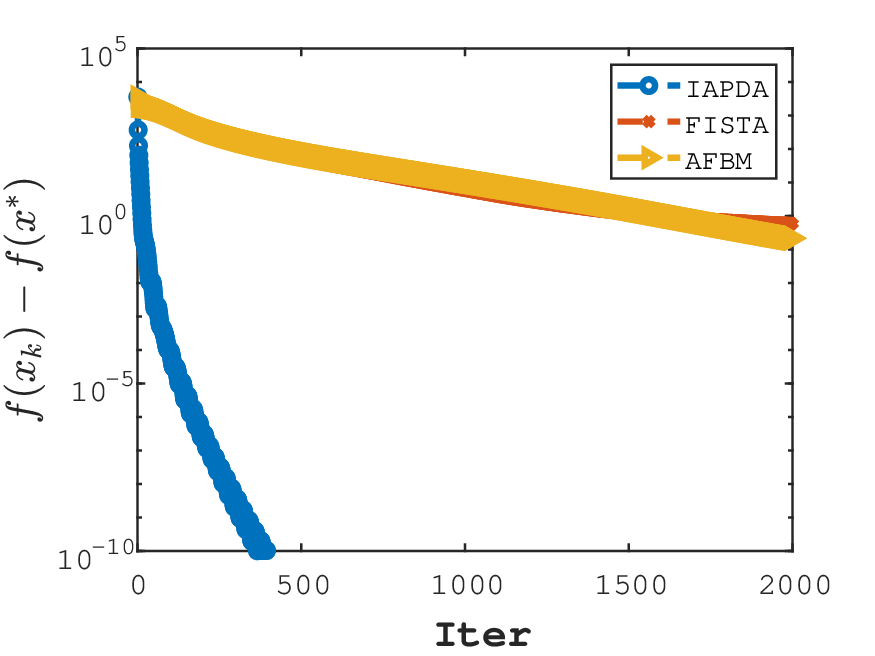}
         }
         \subfloat[$  \gamma=1 $ ]{
         \includegraphics[width=0.48\linewidth]{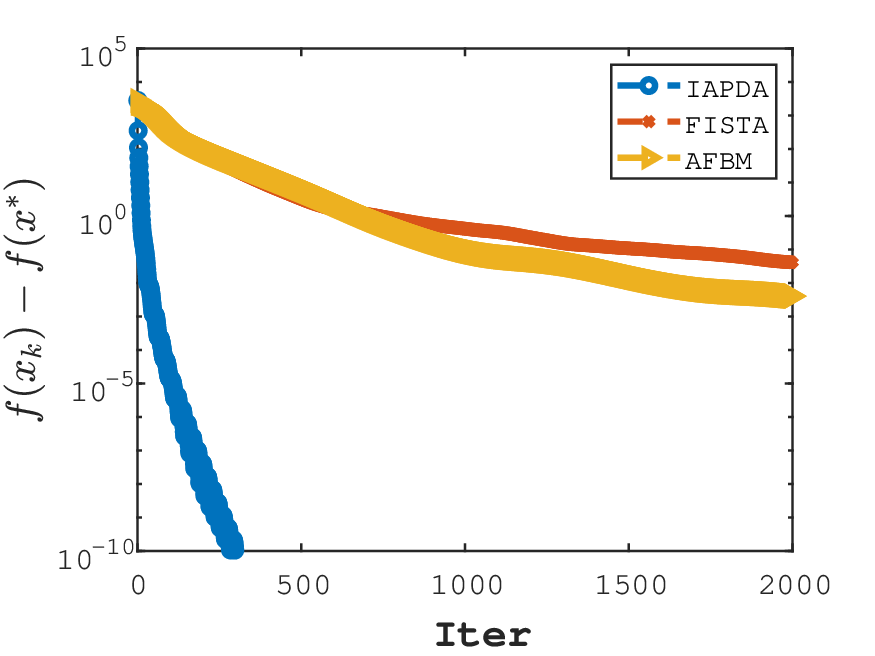}
         }\\
     \caption{Numerical results of IAPDA, FISTA and AFBM when $m=1500$, $n=2000$.}
     \label{HD-exp4}
 \end{figure}
\end{example}

\section{Conclusion}
In this paper, we proposed the IAPDA derived from the time discretization of a second-order differential system (\ref{dyn}), for solving non-smooth convex optimization problems with linear equality constraints. Then, we established convergence rates for the primal-dual gap, the feasibility violation, and the objective residual. Numerical experiments on $ \ell_1-\ell_2 $ minimization problem and the nonnegative least squares problem demonstrated the effectiveness and superior performance of IAPDA over existing state-of-the-art methods.

\textcolor{red}{Despite these advances, the convergence analysis of trajectories generated by IAPDA remains an open challenge. In future work, we aim to address this limitation by reformulating the system and constructing an appropriate energy function, with the goal of rigorously establishing trajectory convergence. }Moreover, the framework presented here can be extended to develop accelerated algorithms for separable convex optimization problems. We also plan to explore accelerated algorithms from the time discretization of the same differential system incorporating Tikhonov regularization to solve convex-concave bilinear saddle point problems.

\section*{Acknowledgements}
{\small The authors would like to express their sincere thanks to the   anonymous reviewers for the constructive suggestions and comments.}

\section*{Funding}
\small{  This research is supported by the Natural Science Foundation of Chongqing (CSTB2024NSCQ-MSX0651) and the Team Building Project for Graduate Tutors in Chongqing (yds223010).}

\section*{Data availability}

 \small{ The authors confirm that all data generated or analysed during this study are included in this article.}

 \section*{Declaration}

 \small{\textbf{Conflict of interest} No potential conflict of interest was reported by the authors.}

\bibliographystyle{plain}

\end{document}